\documentclass[11pt,letterpaper]{amsart}

\usepackage[margin=1.5in]{geometry}
\usepackage{amssymb,amscd, xypic, amsmath,amsthm, color,tabmacD,mathrsfs,bm}
\usepackage[pdfencoding=auto, psdextra, colorlinks=true, linkcolor = blue, urlcolor=black, citecolor=blue, bookmarksdepth=10]{hyperref}
\usepackage{amsfonts}
\usepackage[shortlabels]{enumitem}
\usepackage{tikz}
\usepackage{latexsym}
\usepackage{mathdots}
\usepackage{mathtools}
\usepackage{tikz-cd}

\usepackage{ytableau}
\ytableausetup{smalltableaux}

\usepackage{multicol}

\usepackage{caption}

\theoremstyle{plain}
\newtheorem{theorem}{Theorem}[section]

\newtheorem{lemma}[theorem]{Lemma}

\newtheorem{claim}{Claim}  

\newtheorem*{thm*}{Theorem}
\newtheorem*{claim*}{Claim}

\theoremstyle{definition}

\newtheorem{example}[theorem]{Example}

\newtheorem{remark}[theorem]{Remark}

\newcommand{\defin}{\textit}
\newcommand{\excise}[1]{}

\newcommand{\Ess}{\mathscr{E}\!ss}       
\newcommand{\esse}{\mathfrak{e}}         

\newcommand{\triple}{{\bm{\tau}}}
\newcommand{\isom}{\cong}
\renewcommand{\bar}{\overline}
\renewcommand{\tilde}{\widetilde}

\renewcommand{\setminus}{\smallsetminus}

\newcommand{\CC}{\mathbb{C}}
\newcommand{\KK}{\mathbb{K}}       

\newcommand{\mult}{\mathrm{mult}}

\newcommand{\e}{\pmb{e}}
\newcommand{\bk}{\pmb{k}}    
\newcommand{\bp}{\pmb{p}}
\newcommand{\bq}{\pmb{q}}
\newcommand{\br}{\pmb{r}}

\newcommand{\EYD}{\mathcal{E}}

\newcommand{\fp}{p}              

\newcommand{\Tableau}[2][sY]{{\text{\tableau[#1]{#2}}}}

\usetikzlibrary{matrix,decorations.pathreplacing}
\usetikzlibrary{arrows,matrix,positioning}

\begin{document}

\author{David Anderson, Takeshi Ikeda, Minyoung Jeon, Ryotaro Kawago}

\address{Department of Mathematics, The Ohio State University, Columbus OH 43210, USA}
\email{anderson.2804@math.osu.edu}
\email{jeon.163@buckeyemail.osu.edu}

\address{Department of Mathematics, Faculty of Science and Engineering, Waseda University, 3-4-1 Okubo, Shinjuku, 
Tokyo 169-8555
JAPAN}
\email{gakuikeda@waseda.jp}
\email{kawa3@akane.waseda.jp}

\title[Multiplicities of vexillary Schubert varieties]
{The multiplicity of a singularity in a vexillary Schubert variety}

\date{December 9, 2021}

\thanks{DA and MJ were partially supported by NSF CAREER DMS-1945212.}
\thanks{TI was partially supported by Grant-in-Aid for Scientific Research(C) 16H03920, 17H02838, 18K03261, 20K03571, 20H00119.}

\begin{abstract}
In a classical-type flag variety, we consider a Schubert variety associated to a vexillary (signed) permutation, and establish a combinatorial formula for the Hilbert-Samuel multiplicity of a point on such a Schubert variety.  The formula is expressed in terms of {\it excited Young diagrams}, and extends results for Grassmannians due to Krattenthaler, Lakshmibai-Raghavan-Sankaran, and for the maximal isotropic (symplectic and orthogonal) Grassmannians to Ghorpade-Raghavan, Raghavan-Upadhyay,  Kreiman, and Ikeda-Naruse. We also provide a new proof of a theorem of Li-Yong in the type A vexillary case.

The main ingredient is an isomorphism between certain neighborhoods of fixed points, known as Kazhdan-Lusztig varieties, which, in turn, relies on a direct sum embedding previously used by Anderson-Fulton to relate vexillary loci to Grassmannian loci. 
\end{abstract}

\maketitle

\section{Introduction}

Singularities of Schubert varieties have been the subject of much study for decades.  At a coarse level, there are efficient criteria for determining whether a given Schubert variety is smooth or singular (e.g., via pattern avoidance); see \cite{BL} for a detailed survey of this and related problems.  In sufficiently large flag varieties, nearly all Schubert varieties are singular, so one is led to consider the question of finer invariants of singularities.  For instance, one can ask about the {\it Hilbert-Samuel multiplicity} of a point.  This is a positive integer which measures singularities: it equals $1$ if and only if the point is nonsingular, and larger multiplicities correspond to more complicated singularities.  

Schubert varieties $\Omega_w$ and torus-fixed points $\fp_v$ are both indexed by the Weyl group, with $\fp_v\in\Omega_w$ iff $w\leq v$ in Bruhat order.  So for any pair $(w,v)$ with $w\leq v$, a natural problem arises:
\begin{quotation}
 {\it Find an explicit formula for the positive integer $\mult_{\fp_v}\Omega_w$. }
\end{quotation}
(Using a Borel group action, any point $x\in\Omega_w$ may be translated to some torus-fixed point $\fp_v$, so answering this question for fixed points solves the problem for arbitrary points.)

The main goal of this article is to solve the above problem in the case where the ambient flag variety is of classical type, and the Schubert variety $\Omega_w$ is indexed by a {\it vexillary} element of the Weyl group (in the sense of \cite{AF18}, see \S\ref{s.vexillary} for the definition).

In the case of Grassmannians---including the (co)minuscule Grassmannians of maximal isotropic subspaces in types B, C, and D---several combinatorial, determinantal, and Pfaffian formulas are known, and due to many mathematicians. In type A, Lakshmibai-Weyman \cite{LW}, and Rosenthal-Zelevinsky gave a determinantal formula \cite{RZ}, while Krattenthaler \cite{Kra} and Lakshmibai-Raghavan-Sankaran \cite{LRS} gave combinatorial formulas, in terms of non-intersecting lattice paths.  The last of these results was reinterpreted by Kreiman \cite{Kre1} and Ikeda-Naruse \cite{IN} in terms of combinatorial objects called {\it excited Young diagrams}.  
In other classical types, Pfaffian formulas for maximal isotropic Grassmannians were given by Ikeda \cite{Ikeda} and Ikeda-Naruse \cite{IN}.  Combinatorial formulas for the Lagrangian Grassmannian were given by Ghorpade-Raghavan \cite{GR}, Kreiman \cite{Kre2}, and Ikeda-Naruse \cite{IN} independently. For the maximal orthogonal Grassmannian, combinatorial formulas were given by Raghavan-Upadhyay \cite{RU} and Ikeda-Naruse \cite{IN}.

Some of these authors (\cite{Ikeda,IN,Kre1,Kre2,LRS}) used torus actions to compute an {\it equivariant} multiplicity; in the (co)minuscule case, this leads directly to a formula for the usual multiplicity, e.g., as explained in \cite[Proposition 9.1]{IN}.  Outside of the cominuscule setting, this technique is not available, because the ideals are no longer homogeneous with respect to the natural torus action.  Partly for this reason, less is known beyond the Grassmannian case.

The class of {\it vexillary permutations} is a natural generalization of Grassmannians.  These were Lascoux and Sch\"utzenberger in the context of determinantal formulas for Schubert polynomials in type A.  An analogous notion for signed permutations (in the other classical types) was introduced by Billey and Lam \cite{BL}; this was later revisited in a geometric context, better adapted to our situation \cite{AF18}.

Li and Yong took this first step beyond Grassmannians: in the type A flag variety, they computed the multiplicity of a Schubert variety indexed by a vexillary permutation \cite{LY}.  Their methods involve a detailed analysis of the ideal defining such a Schubert variety, and include a Gr\"obner basis for such ideals.  Their combinatorial language is that of {\it flagged semistandard tableaux}, which is natural bijection with others mentioned above (lattice paths, excited Young diagrams) and also admits a determinantal formula. 

The type A results of Li and Yong inspired our investigation of vexillary Schubert varieties in the other classical types.  
To state the theorem, we need some notation, which will be reviewed in more detail below (see \S\ref{s.vexillary}).  Each vexillary permutation $w$ comes with a partition $\lambda$, and to any $v\geq w$ we associate an {\it outer shape} $\mu \supset \lambda$.  An {\it excitation} (or {\it excited Young diagram}) of $\lambda$ in $\mu$ is a collection of boxes $C \subset \mu$ which are obtained from those of $\lambda$ by a sequence of certain local moves (whose precise description depends on type).  We write $\EYD_\mu(\lambda)$ for the set of all such excitations.

Here is an example in type A, where a permutation is vexillary if and only if it avoids the pattern $2\,1\,4\,3$, and the local moves generating excited Young diagrams are of the form
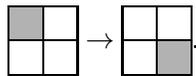
\[
\Tableau{\gray&\\&}\rightarrow
\Tableau{&\\&\gray}.
\]
For the vexillary permutation $w=1\,2\,5\,4\,3\,6\,7$, the corresponding partition is $\lambda=(2,1)$.  The permutation $v=5\,2\,6\,4\,1\,7\,3$ is above $w$ in Bruhat order, and the associated outer shape is $\mu=(3,3,2)$.  Then $\EYD_\mu(\lambda)$ consists of the $5$ diagrams
$${\small
\Tableau{\gray&\gray&\\ \gray&&\\&}\quad\;
\Tableau{\gray&&\\ \gray&&\gray\\&}\quad\;
\Tableau{\gray&\gray&\\&&\\&\gray}\quad\;
\Tableau{\gray&&\\&&\gray\\&\gray}\quad\;
\Tableau{&&\\&\gray&\gray\\&\gray}.}
$$

\begin{thm*}
Let $w$ be a vexillary (signed) permutation, and $v$ any element such that $w\leq v$.  Let $\fp_v\in\Omega_w$ be the corresponding fixed point and Schubert variety, and let $\lambda\subset \mu$ the corresponding partition and outer shape.  Then
\[
  \mult_{\fp_v}\Omega_w = \#\EYD_\mu(\lambda),
\]
where $\EYD_\mu(\lambda)$ is the set of excited states of $\lambda$ inside $\mu$.
\end{thm*}

\noindent
In the setting of the above example, the theorem says $\mult_{\fp_v}\Omega_w = 5$.

Our approach yields an {\it a priori} reason for the fact that multiplicities of points on vexillary Schubert varieties are computed by a formula which also computes multiplicities on Grassmannian Schubert varieties; in particular, we provide an alternative proof of \cite[Theorem~6.2]{LY}.  
The main innovation is an explicit isomorphism, up to irrelevant factors of affine spaces, between canonical affine neighborhoods of fixed points in Schubert varieties in the flag variety and Grassmannian.  These neighborhoods are sometimes called {\it Kazhdan-Lusztig varieties}, and they have been studied by many authors \cite{EFRW,K,LY2,LY,WY}; see \S\ref{s.KLisom} for details.
  
The isomorphism is set up using a direct-sum embedding of a flag variety in a larger one.  When $\Omega_w$ is a vexillary Schubert variety in the (ordinary or isotropic) flag variety for a vector space $V$, we construct local isomorphisms with the Schubert variety $\Omega_\lambda$ in the Grassmannian of (ordinary or isotropic) half-dimensional subspaces of $V\oplus V$.  This allows us to reduce the computation of local invariants of points on vexillary Schubert varieties to a known calculation on Grassmannian Schubert varieties.  We expect this method will find further use; see \cite{J} for an application to Kazhdan-Lusztig polynomials.

The arguments proceed in parallel for all classical types.  All the essential ideas appear in type A, so the reader is encouraged to digest that case first.

One can consider more refined invariants, such as the Hilbert series.  For Schubert varieties in any (co)minuscule Grassmannian---including types $E_6$ and $E_7$, as well as the maximal Grassmannians in classical types---Lakshmibai and Weyman gave a positive recursive formula for the Hilbert series and multiplicity \cite{LW}.  For classical-type Grassmannians, the Hilbert series was computed explicitly by Kodiyalam-Raghavan \cite{KR} in type A, Ghorpade-Raghavan \cite{GR} in type C, and Raghavan-Upadhyay \cite{RU} in type D.  Using equivariant K-theory and variations on excited Young diagrams, these invariants were also studied by Graham-Kreiman \cite{GrKr}.  For vexillary Schubert varieties in type A, the Hilbert series was studied by Li-Yong in connection with Kazhdan-Lusztig polynomials \cite{LY2}.  It would be interesting to see analogues of these results for vexillary Schubert varieties in other classical types.

Some of our results were announced at FPSAC XXXI (2019) and appeared in the proceedings of that meeting \cite{AIJK19}.  The proof sketch given there (\cite{AIJK19}, \S 4) contains an incorrect assertion and a major gap; however, it is replaced by the significantly stronger isomorphism of Kazhdan-Lusztig varieties given here in \S\ref{s.KLisom}.

\medskip
\noindent
{\it Acknowledgements.}   TI is particularly indebted to K. N. Raghavan and Vijay Ravikumar for stimulating discussions in the early stage of the project.  We also thank Tomoo Matsumura for valuable discussions.

\section{Notation and Definitions}

We work over an algebraically closed field $\KK$ of characteristic\footnote{We rely on some sources where results  are stated over $\CC$---e.g., \cite{IN} and \cite{Kre1,Kre2}---but the same proofs work over arbitrary algebraically closed fields.  See \cite{LW}, which implies that for (co)minuscule Schubert varieties, the multiplicity is independent of characteristic; the question of whether this is true in general is raised at the end of \cite[\S8]{LY}.  In type A, our methods work in arbitrary characteristic, but in other types, we must avoid characteristic $2$.} not equal to $2$.

Let $G$ be a semisimple linear algebraic group defined over $\KK$.  Fix a maximal torus $T$ and a Borel subgroup $B$ such that $T\subset B\subset G$.  Let $W=N_G(T)/T$ be the Weyl group.  The fixed points of $G/B$ under the left action of $T$ are naturally indexed by the Weyl group $W$ of $G$ with respect to $T$: given $w\in W$, we write $\fp_w\in G/B$ for the associated fixed point.

Let $B_{-}$ be the opposite Borel subgroup, so $B\cap B_{-}=T$.  The {\it Schubert 
cell} $\Omega_w^\circ$ is the $B_{-}$ orbit $B_{-}\cdot\fp_w$.  The closure of a Schubert cell  is the {\it Schubert variety} $\Omega_w=\overline{\Omega_w^\circ}$, a subvariety of codimension $\ell(w)$, where $\ell(w)$ is the length of $w$.  The Weyl group is partially ordered by {\it Bruhat order}, defined by $w\leq v$ if and only if $\fp_v\in \Omega_w$.  
The Schubert variety $\Omega_w$ is a disjoint union of Schubert cells $\Omega_v^\circ$ such that $v\geq w$.

We will also occasionally consider {\it opposite Schubert cells} $X_w^\circ = B\cdot \fp_w$.    These are affine spaces of dimension $\ell(w)$.

In what follows, we quickly review type-specific notation for Schubert varieties in Grassmannians and flag varieties.  More discussion can be found in sources such as \cite{FP}.  In each case, after fixing a basis, $G$ will be a particular matrix group, $T$ will be diagonal matrices in $G$, $B$ upper-triangular matrices in $G$, and $B_-$ lower-triangular matrices in $G$.

\subsection*{Type A}

Let $V$ be a vector space of dimension $n$ with basis $\e_1,\ldots,\e_n$, and $G=SL(V)=SL_n$.    
The Weyl group is $S_n$, the symmetric group of permutations of $[n] := \{1,\ldots,n\}$.  We  write permutations $w\in S_n$ in one-line notation, by recording values: $w = w(1)\;w(2)\cdots w(n)$.

The flag variety $G/B$ is the variety $Fl(V)$ of complete flags 
$$E_\bullet: 
\{\pmb{0}\}\subset E_1\subset E_2\subset \cdots
\subset E_n=\KK^n,\quad
\dim(E_i)=i\quad(1\leq i\leq n).
$$
For $w\in S_n$ the corresponding $T$-fixed point $\fp_w$ is given by  $E_i=\mathrm{span}\{\e_{w(1)},\ldots,\e_{w(i)}\}.$

The $B_-$-fixed flag $F^\bullet$ is given by subspaces $F^i = \mathrm{span}\{\e_n,\ldots,\e_{i+1}\}$, of codimension $i$.  Schubert varieties are defined by intersection conditions with $F^\bullet$, as follows.  For each $w\in S_n$, the function $k_w$ is defined on the $n\times n$ grid by 
$$
k_w(q,p)=\#\{s\in [n]\;|\;s\leq p,\;w(s)> q\},\quad (q,p)\in [n]\times [n].
$$
The Schubert variety is 
\begin{equation}
\Omega_w=\{E_\bullet\;|\; \dim(E_{p}\cap F^q)\geq k_w(q,p) \;\mbox{for all}\;
1\leq q,p\leq n \}. \label{eq:DefOmage_w}
\end{equation}
An alternative description is
$$
\Omega_w=\{E_\bullet\;|\;
\mathrm{rank} (E_{p}\rightarrow V/F^q)\leq r_w(q,p)
\;\mbox{for all}\;
1\leq q,p\leq n
\},$$
where 
\begin{align*}
r_w(q,p) &=\#\{s\in [n]\;|\;s\leq p,\;w(s)\leq q\},\quad (q,p)\in [n]\times [n] \\
         &= p-k_w(q,p).
\end{align*}

Let $d$ be an integer such that $1\leq d\leq n=\dim(V)$.
In the Grassmannian $Gr(d,V)$ of $d$-dimensional subspaces, $T$-fixed points and Schubert varieties are indexed by partitions $\lambda$ whose Young diagram fits inside the $d\times (n-d)$ rectangle; that is,
\[
  \lambda = (n-d \geq \lambda_1 \geq \cdots \geq \lambda_d \geq 0 ).
\]
Each such $\lambda$ determines a {\it Grassmannian permutation} $w_\lambda$, with unique descent at $d$, as follows.  Let $I(\lambda) = \{ i_1 < \cdots < i_d \}$ be defined by $i_k = k+\lambda_{d+1-k}$, and let $J(\lambda) = \{ j_1 < \cdots < j_{n-d} \} = [n]\setminus I(\lambda)$ be the complement.  Then
\[
  w_\lambda = i_1 \cdots i_d\; j_1 \cdots j_{n-d}.
\]
This has length is equal to the number of boxes in $\lambda$, that is, $\ell(w_\lambda) = |\lambda|:= \sum \lambda_k$. 

We refer the reader to the example in the introduction and \cite[Section 3.3]{IN} for the relation between a partition $\lambda$ and its Young diagram as well as $w_\lambda$.

The $T$-fixed point $\fp_\lambda\in Gr(d,V)$ corresponds to the subspace spanned by $\{e_i \,|\, i\in I(\lambda)\}$.  This is the same as the $d$-dimensional component of the $T$-fixed flag corresponding to $w_\lambda$.

The Schubert variety $\Omega_\lambda \subset Gr(d,V)$ is the closure of the Schubert cell $\Omega_\lambda^\circ = B_-\cdot \fp_\lambda$.  It can be described as
\[
  \Omega_\lambda = \{ E \in Gr(d,V) \,|\, \dim( E \cap F^{\lambda_k+d-k} ) \geq k \text{ for } 1\leq k \leq d \}.
\]
This has codimension $|\lambda|$ in $Gr(d,V)$.

The projection $\pi_d\colon Fl(V) \to Gr(d,V)$ which sends a flag $E_\bullet$ to $E_d$ is a locally trivial $G$-equivariant fiber bundle, with smooth fibers; the fiber over $E_d \subset V$ is naturally identified with $Fl(E_d) \times Fl(V/E_d)$.  Furthermore, one has
\[
  \Omega_{w_\lambda} = \pi_d^{-1}\Omega_\lambda,
\]
so in particular $\Omega_{w_\lambda}$ is a smooth fiber bundle over $\Omega_\lambda$.

\subsection*{Type C}

Let $V$ be a vector space of dimension $2n$ with basis $\e_{\bar n},\ldots,\e_{\bar 1},\e_1,\ldots,\e_{n}$.  Let us define a skew-symmetric bilinear form on $V$ by 
$\langle\e_i,\e_j\rangle=\langle\e_{\bar \imath},\e_{\bar \jmath}\rangle
=0$ and $\langle\e_i,\e_{\bar \jmath}\rangle
=-\langle\e_{\bar \jmath},\e_i\rangle=\delta_{i,j}$
for $1\leq i,j\leq n.$  (The bar is meant to indicate a negative number, so $\bar{a} = -a$.)  The Gram matrix looks like this:
\[
\left[ \begin{array}{ccc|ccc}
 & & & & &  1 \\
 & & & &  \iddots &\\
 & & &  1 & &  \\ \hline
 & & -1 & & &  \\
 & \iddots & & & &  \\ 
-1 & & & & &  \\
\end{array} \right].
\]
Let $G=Sp(V)=Sp_{2n}$ be the symplectic group with respect to the form. As we noted above, 
we choose a torus $T$ consisting of diagonal matrices in $Sp_{2n}(\KK)$, a Borel group 
$B$ of the upper triangular matrices in $Sp_{2n}(\KK)$, and an Borel subgroup $B_-$ opposite to $B$ of the lower-triangular matrices.  The Weyl group is the group of signed permutations $W_n=S_n\ltimes \{\pm 1\}^n$.  This is often realized as the subgroup $W_n$ of $S_{2n}$ such that $w(\bar\imath)=\bar{w(i)}$ for all $i$. Here we consider $S_{2n}$ as the group of permutations of $\{\pm1,\ldots,\pm n\}.$   We usually write a signed permutation $w$ in ``one-line notation'' $w(1)\cdots w(n)$ (note that we only need to record $w(i)\,(i\in [n])$.

An {\it isotropic flag} is a sequence of linear subspaces 
$$E_\bullet: 
\{\pmb{0}\}\subset E_{n}\subset E_{n-1}\subset \cdots
\subset E_{1}\subset V\cong\KK^{2n},\quad
\mathrm{dim}(E_{i})=n+1-i\quad(1\leq i\leq n),
$$
where $E_i$ are all isotropic.  
In particular, $E_1$ is a Lagrangian subspace, i.e., a maximal isotropic subspace of $V$.  The variety $Fl^C(V) = Sp_{2n}/B$ is identified with the set of all isotropic flags.

The unique $B_-$-fixed isotropic flag is given by $F_q = \mathrm{span}\{\e_n,\ldots,\e_{q}\}$ for $1\leq q\leq n$.  This is extended to a complete flag by setting
\[
F_{\bar q} = F_{q+1}^\perp = \mathrm{span}\{ \e_n,\ldots,\e_1,\e_{\bar1},\ldots,\e_{\bar{q}} \}
\]
for $1\leq q\leq n$.

Given $w\in W_n$, the Schubert variety is described as
$$\Omega_w=\{E_\bullet\;|\;\mathrm{dim}(E_p\cap F_q)\geq k_w(q,\bar{p})\;\text{for all}\;  p \in \{1,\ldots,n\},\, q\in \{\bar{n},\ldots,\bar{1},1,\ldots,n\} \},$$ 
where the function $k_w$ associated to $w$ is defined by
\[
  k_w(q,\bar{p})=\#\{s \;|\;s\leq \bar{p},\;w(s) \geq  q\}.
\]

As in type A, the Schubert variety can also be written
\[
\Omega_w=\{E_\bullet\;|\;\mathrm{rank}(E_p\rightarrow V/F_q)\leq r_w(q,\bar{p})\}
\]
for all $1\leq p\leq n$ and $q\in \{\bar{n},\ldots,\bar{1},1,\ldots,n\}$.  
Here
\begin{align*}
r_w(q,\bar{p})&=\#\{s\in \{\bar{n},\ldots,\bar{1}\} \;|\;s\leq \bar{p},\;w(s) < q\} \\
 &=n+1-p-k_w(q,p).
 \end{align*}  

We write $LG(V)$ for the Lagrangian Grassmannian parametrizing $n$-dimensional isotropic subspaces of $V$. The index set for the Schubert varieties and $T$-fixed points is the set of strict partitions
$\lambda=(\lambda_1,\ldots,\lambda_n)$
satisfying 
\begin{equation}\label{lamcon}
n\geq \lambda_1> \cdots> \lambda_r>0,\;
\mbox{and}\;
\lambda_i=0
\;\mbox{for}\;r<i\leq n
\end{equation}
for some $r\leq n$.  Equivalently, $\lambda$ is an $r$-element subset of $\{1,\ldots,n\}$.   We denote by $r=\ell(\lambda)$, the length of $\lambda$. 
Given such a $\lambda$, the corresponding
Schubert variety is 
\begin{equation}
\Omega_\lambda=\{
W\in LG(V)\;|\; \dim(W\cap F_{\lambda_i})\geq i
\;
\mbox{for}\;
1\leq i\leq \ell(\lambda)
\}, \label{eq:OmegaForLG}
\end{equation}
of codimension $|\lambda|$.

Thinking of $\lambda$ as a subset of $\{1,\ldots,n\}$, the fixed point $\fp_\lambda$ corresponds to the isotropic subspace spanned by $\e_i$ for $i \in \lambda$, together with $\e_{\bar\imath}$ for $i\not\in\lambda$.  For example, $\fp_\emptyset = \mathrm{span}\{ \e_{\bar{n}},\ldots,\e_{\bar{1}}\}$, and $\fp_{(n,\ldots,2,1)} = F_1 = \mathrm{span}\{ \e_1,\ldots,\e_n\}$.  The Schubert variety $\Omega_\lambda$ is the closure of the Schubert cell $\Omega_\lambda^\circ=B_{-}\cdot \fp_\lambda$.

Each strict partition $\lambda=(\lambda_1,\ldots,\lambda_r)$ corresponds to a {\it Grassmannian signed permutation} 
$$w_\lambda=\bar{\lambda}_1\cdots\bar{\lambda}_rj_1\cdots j_{n-r},$$
where $\{j_1<\cdots<j_{n-r}\} = [n]\backslash \{\lambda_1,\ldots,\lambda_r\}$. See, e.g., \cite[Section 3.4]{IMN}.

Let $\pi\colon Fl^C(V) \to LG(V)$ be the projection sending an isotropic flag $E_\bullet$ to the maximal isotropic subspace $E_1$.  As for type A, this is a $G$-equivariant fiber bundle, where each fiber is smooth and the fiber over $E_1$ is isomorphic to the flag variety $Fl(E_1)$.  We have
$\pi^{-1}\Omega_\lambda = \Omega_{w_\lambda}$, and $\Omega_{w_\lambda}$ is a fiber bundle over the Schubert variety $\Omega_\lambda$.

\subsection*{Type B}

Let $V$ be a vector space of dimension $2n+1$,
with basis $\e_{\bar n},\ldots,\e_{\bar 1},\e_0,\e_1,\ldots,\e_{n}$. 
We have a symmetric bilinear form on $V$, defined by 
$\langle\e_i,\e_j\rangle=\langle\e_{\bar i},\e_{\bar j}\rangle
=\langle\e_0,\e_j\rangle
=\langle\e_0,\e_{\bar j}\rangle=0$, $\langle\e_i,\e_{\bar j}\rangle
=\delta_{i,j}$
for $1\leq i,j\leq n$, and $\langle\e_0,\e_{0}\rangle=1.$  So its Gram matrix looks like this:
\[
\left[ \begin{array}{ccc|c|ccc}
 & & & & & & 1 \\
 & & & & & \iddots &\\
 & & & & 1 & &  \\ \hline
 & & & 1 & & & \\ \hline
 & & 1 & & & & \\
 & \iddots & & & & & \\ 
1 & & & & & & \\
\end{array} \right].
\]
Let $G=SO(V)=SO_{2n+1}$ be 
the special orthogonal group with respect to the form.
The Weyl group is the same as in type C, the group of signed permutations $W_n=S_n\ltimes \{\pm 1\}^n$.

The isotropic flag variety $Fl^B(V)=SO_{2n+1}/B$ consists of a sequence of linear subspaces 
\[
E_\bullet: 
\{\pmb{0}\}\subset E_{n}\subset E_{n-1}\subset \cdots
\subset E_{1}\subset V\cong\KK^{2n+1},\quad
\dim(E_{i})=n+1-i\quad(1\leq i\leq n),
\]
where $E_i$ are all isotropic.  
In particular,
$E_1$ is a maximal isotropic subspace of $V$.  
Let us define
\[
F_i=\mathrm{span}_{\KK}
\{\e_{i},\ldots,\e_n
\}\quad \mbox{for}\;1\leq i\leq n.
\]
Then $F_\bullet=(F_1\supset \cdots\supset F_n)$ is the unique  
isotropic flag fixed by $B_{-}$.  
Note that $F_1$ is a maximal isotropic subspace.

Inside the odd Orthogonal Grassmannian $OG(n,V)$ of $n$-dimensional isotropic subspaces of $V$, the Schubert variety associated with $\lambda$ is given by the same conditions (\ref{eq:OmegaForLG}) as in type C.  The correspondence between $T$-fixed points of $OG(n,V)$ and strict partitions inside the $n$-staircase is also the same as in type C.

\subsection*{Type D}

We consider a vector space $V$ of dimension $2n$
with a basis indexed by the set
$\{\bar{n},\ldots,\bar{1},1,\ldots, n\}$. 
Let us define a symmetric bilinear form on $V$ by 
$\langle\e_i,\e_j\rangle=\langle\e_{\bar i},\e_{\bar j}\rangle
=0$, $\langle\e_i,\e_{\bar j}\rangle
=\delta_{i,j}$
for $1\leq i,j\leq n$.  Its Gram matrix is
\[
\left[ \begin{array}{ccc|ccc}
 & & & & &  1 \\
 & & & &  \iddots &\\
 & & &  1 & &  \\ \hline
 & & 1 & & &  \\
 & \iddots & & & &  \\ 
1 & & & & &  \\
\end{array} \right].
\]
Let $G=SO(V)=SO_{2n}$ be 
the special orthogonal group with respect to the form.
The Weyl group is the subgroup $W_{n}^+ \subset W_{n}$ 
consisting of signed permutations with even number of sign changes.

Let $Fl^D(V)=SO_{2n}/B$.  Here an {\it isotropic flag\/} is a chain of subspaces 
\[
E_\bullet: 
\{\pmb{0}\}\subset E_{n-1}\subset \cdots \subset E_{1}
\subset E_{0}\subset V\cong\KK^{2n},
\]
with $\dim(E_{i})=n-i$ for $1\leq i\leq n$, where $E_{0}$ is a maximal isotropic subspace of $V$.  
We fix $F_\bullet$ given by 
$F_q=\mathrm{span}_{\KK}\{\e_{q+1},\ldots,\e_{n}\}$ for $0\leq q\leq n-1$
as the reference isotropic flag, extended to a complete flag by $F_{\bar{q}} = F_q^\perp$ for $1\leq q\leq n$.  (Here $F_n=0$.)

The type D flag variety $Fl^D(V)$ is the set of all isotropic flags in $V$, with the additional condition that $\dim(E_0 \cap F_0)$ is even.

Given $w\in W_n^+$, the corresponding Schubert variety in $\Omega_w \subset Fl^D(V)$ is the closure of the cell
\[
 \Omega_w^\circ = \big\{ E_\bullet \,|\, \dim(E_p\cap F_q) = k_w(q,\bar{p}) \text{ for } 0\leq p \leq n-1 \text{ and } q\in\{\bar{n-1},\ldots,\bar{0},0,\ldots,n-1\} \big\},
\]
where now $k_w(q,\bar{p}) = \#\{ s< \bar{p} \,|\, w(s)>q \}$.

The (even) orthogonal Grassmannian $OG^+(n, V)$ is the set of maximal isotropic subspaces $W$ in $V$ such that $\mathrm{dim}(W\cap F_0)$ is even.  
For a strict partition $\lambda$, the corresponding Schubert variety in $OG^+(n,V)$ is the closure of
\[
\Omega^\circ_\lambda=\{W\in OG^+(n,2n)
\;|\;
\dim(W\cap F_{\lambda_i}) = i\;\mbox{for}\;1\leq i\leq \ell(\lambda)
\}.
\]
Its codimension in $OG^+(n,2n)$ is $|\lambda|$.

Fixed points in $OG^+(n,2n)$ are parametrized as follows.  Let $\lambda=(\lambda_1>\cdots>\lambda_r\geq 0)$ be a strict partition.  We may ensure $r=\ell(\lambda)$ is even by including or omitting a $0$ part, as needed.  The fixed point $\fp_\lambda$ is the maximal isotropic subspace spanned by $\e_i$ for $i\in\{\lambda_1,\ldots,\lambda_r\}$ and $\e_{\bar{\imath}}$ for $i\in\{0,\ldots,n\}\setminus\{\lambda_1,\ldots,\lambda_r\}$.

For instance, we have a $T$-fixed point $\fp_\emptyset=\mathrm{span}\{\e_{\bar{n}},\ldots,\e_{\bar{1}}\}$.  If $n$ is even, then $\fp_{(n-1,\ldots,1)}=\fp_{(n-1,\ldots,1,0)}=\mathrm{span}\{\e_1, \e_2,\ldots,\e_n\}$; if $n$ is odd, then $\fp_{(n-1,\ldots,1)} = \mathrm{span}\{\e_{\bar{1}},\e_2,\ldots,\e_n\}$.

Given $\lambda=(\lambda_1>\cdots>\lambda_r\geq 0)$ with $r$ even, set $\lambda^+ = (\lambda_1+1>\cdots>\lambda_r+1)$.  The type D Grassmannian signed permutation associated to $\lambda$ is the type C one associated to $\lambda^+$, i.e.,
\[
 w_\lambda = \bar{\lambda_1+1} \cdots \bar{\lambda_r+1} \, j_1 \cdots j_{n-r},
\]
where $\{j_1<\cdots<j_{n-r}\} = [n]\setminus \{\lambda_1+1,\ldots,\lambda_r+1\}$.

\section{Vexillary (signed) permutations and Schubert varieties}\label{s.vexillary}

Here we review the notions we need related to vexillary permutations.  For type A, this is standard by now (see, e.g., \cite{F1,Mac}).  For other types, vexillary elements were first considered by Billey and Lam \cite{BLam}.  We use a variation introduced in \cite{AF12,AF18}, which is adapted to the geometry of Schubert varieties: in each type, a vexillary (signed) permutation $w$ corresponds to a triple $(\bk,\bp,\bq)$, which in turn records a set of {\it essential conditions} defining the Schubert variety $\Omega_w$.

We also recall the {\it shape} $\lambda$ of a vexillary element $w$.  What is new here is the notion of an {\it outer shape} $\mu \supseteq \lambda$ associated to any $v\geq w$.  (In type A, this turns out to recover notation used by Li and Yong \cite{LY}.)

\subsection{Type A}
For $w\in S_n$, there is a subset $\Ess(w)$ of the grid $[n]\times [n]$ of boxes such that
 \begin{equation}
\Omega_w=\{E_\bullet\;|\;
\dim(E_{p}\cap F^q)\geq k_w(q,p)
\;\mbox{for all}\;
(q,p)\in \Ess(w)
\}.\label{eq:OmegaDefEss}
\end{equation}
$\Ess(w)$ is called the {\it essential set\/} of $w$, which is defined as follows.  
The {\it Rothe diagram\/} $D(w)$ of $w$ is a collection of
boxes $(q,p)\in [n]\times [n]$ (arranged as in the matrix indices) 
defined by
\[
D(w)=\{(i,j)\;|\;
i< w(j),\;
j<w^{-1}(i)
\}.
\]  
One way to obtain $D(w)$ is to put dots in the boxes $(w(j),j)$ for $1\leq j\leq n$, and then remove (or shade) all the 
boxes which are south or east of a dot.  
The set of remaining (unshaded) boxes is $D(w).$  Then
$\Ess(w)$ is the set of south-east corners of $D(w)$.
For example, take $w=1\,3\,5\,7\,4\,2\,6${ (which is vexillary)}.  As shown in the following figure, $D(w)$ is the set of unshaded boxes, and $\Ess(w)$ is the set of boxes with stars.  

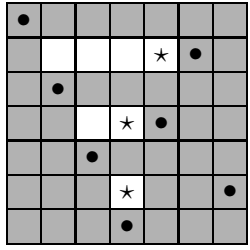
\[
\Tableau{
\gray\bullet&\gray&\gray&\gray&\gray&\gray&\gray\\
\gray&&&&\star&\gray\bullet&\gray\\
\gray&\gray\bullet&\gray&\gray&\gray&\gray&\gray\\
\gray&\gray&&\star&\gray\bullet&\gray&\gray\\
\gray&\gray&\gray\bullet&\gray&\gray&\gray&\gray\\
\gray&\gray&\gray&\star&\gray&\gray&\gray\bullet\\
\gray&\gray&\gray&\gray\bullet&\gray&\gray&\gray\\
}
\]

A permutation $w$ can be recovered by knowing $\Ess(w)$ and the values of
dimension function $k_w$ at all $(q,p)\in \Ess(w)$ (\cite[Lemma 3.10 (a)]{F1}).  We also remark that $w\leq v$ if and only if $k_w(q,p)\leq k_v(q,p)$ for all $(q,p)\in \Ess(w)$.

\subsubsection{Vexillary permutations}

A permutation $w$ is {\it vexillary\/} if the boxes in $\Ess(w)$ can be ordered $\mathfrak{e}_1,\ldots,\mathfrak{e}_s$, proceeding 
(weakly) from south-west to north-east --- that is, if we set $\mathfrak{e}_i=(q_i,p_i)$, then we have 
\begin{equation}
1\leq p_1 \leq \cdots \leq p_s\leq n\quad \mbox{and}\quad n\geq q_1 \geq \cdots \geq q_s\geq 1.
\label{eq:triple1}
\end{equation}
Let $k_i=k_w(\mathfrak{e}_i)\;(1\leq i\leq s).$ Then \eqref{eq:OmegaDefEss} reads
\begin{equation}\label{eq:Omgea_wAtriple}
 \Omega_w = \{ E_\bullet \,|\, \dim(E_{{p}_i} \cap F^{q_i}) \geq k_i \text{ for }i=1,\ldots,s\}.
\end{equation}
 For the above example,
 we have $\mathfrak{e}_1=(6,4),\;
 \mathfrak{e}_1=(4,4),\;
 \mathfrak{e}_1=(2,5)$, and
\[
 \pmb{p}=(4,4,5),\quad
 \pmb{q}=(6,4,2),\quad
 \pmb{k}=(1,2,4).
\]
It is known \cite{AF12} that we have
\begin{eqnarray}
&&0\leq k_1<\cdots<k_s\leq n,\label{eq:triple_k}\\
&&q_i-p_i+k_i>q_{i+1}-p_{i+1}+k_{i+1}
\;\mbox{for}\;1\leq i\leq s-1,\;
\mbox{and}\;
q_{s}-p_{s}+k_{s}>0.\label{eq:triple2}
\end{eqnarray}
If, on the other hand, we have integer vectors 
$\pmb{k},\pmb{p},\pmb{q}$
satisfying \eqref{eq:triple1}, \eqref{eq:triple_k}, \eqref{eq:triple2},
we call $\tau=(\pmb{k},\pmb{p},\pmb{q})$ a {\it triple\/}.
For each triple $\tau$, there is a unique vexillary 
permutation $w$ (see \cite{AF12,AF18}).
We may use the notation $w(\tau)$ to indicate the vexillary permutation corresponding to
a triple $\tau$. 

Let $S_n^\#$ denote the set of all vexillary permutations in $S_n$.  
These can be characterized in terms of $\Ess(w)$, and recovered from the restriction of dimension function to $\Ess(w)$, as follows.  
A permutation $w$ is vexillary if and only if the boxes in $\Ess(w)$ can be ordered $\mathfrak{e}_1,\ldots,\mathfrak{e}_s$, proceeding 
(weakly) from south-west to north-east---that is, $\mathfrak{e}_i=(q_i,p_i)$, with $p_1 \leq \cdots \leq p_s$ and $q_1 \geq \cdots \geq q_s$.
Let $k_i:=k_w(q_i,p_i)$, for $1\leq i\leq s$.  Then $w$ is the unique minimal permutation (in Bruhat order) which has $k_i$ dots strictly south and weakly west of box $\mathfrak{e}_i$, for each $i$.  
 For the above example,
 we have $\mathfrak{e}_1=(6,4),\;
 \mathfrak{e}_1=(4,4),\;
 \mathfrak{e}_1=(2,5)$, and
\[
 \pmb{p}=(4,4,5),\quad
 \pmb{q}=(6,4,2),\quad
 \pmb{k}=(1,2,4).
\]

\subsubsection{Shape of a vexillary permutation}
For each vexillary $w(\tau)$, there is an associated partition $\lambda$, the smallest partition such that 
\[
\lambda_{k_i}=q_i-p_i+k_i\quad(1\leq i\leq s).
\]

The shape $\lambda$ can also be obtained from the diagram.  If we set $r_i=r_w(\mathfrak{e}_i)$, we have $r_i=p_i-k_i$, and the partition $\lambda=\mathrm{sh}(w)$ is obtained as follows: We move each $\mathfrak{e}_i$ diagonally north-west by $r_i$ units, denoted by $\mathfrak{e}_i'$.  Next we make a Young diagram fitting at north-west corner of the $n\times n$ grid having $\mathfrak{e}_i'$'s as corners.  Then $\lambda$ is the transpose (conjugate) of the partition.  For example, we have $\mathrm{sh}(w)=(3,2,1,1)$.
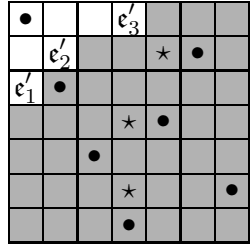
\[
\Tableau{
\bullet&&&\mathfrak{e}_3'&\gray&\gray&\gray\\
&\mathfrak{e}_2'&\gray&\gray&\gray\star&\gray\bullet&\gray\\
\mathfrak{e}_1'&\gray\bullet&\gray&\gray&\gray&\gray&\gray\\
\gray&\gray&\gray&\gray\star&\gray\bullet&\gray&\gray\\
\gray&\gray&\gray\bullet&\gray&\gray&\gray&\gray\\
\gray&\gray&\gray&\gray\star&\gray&\gray&\gray\bullet\\
\gray&\gray&\gray&\gray\bullet&\gray&\gray&\gray\\
}
\]

\subsubsection{Outer shape}

Let $w$ be a vexillary permutation, and let $v$ be a permutation with $w\leq v$.  For each essential box $\esse$ in $\Ess(w)$, let $\esse'$ be the box obtained by moving $\esse$ diagonally north-west by $r_v(\mathfrak{e})$ units.  The {\it outer shape} $\mu$ is the transpose of the smallest Young diagram sitting at north-west corner of the $n\times n$ grid and containing all the boxes $\esse'$.  
For example, consider $w=5\,4\,6\,1\,7\,2\,8\,3$ and  $v=8\,7\,6\,5\,4\,3\,2\,1\geq w$. Then we obtain the outer shape $\mu=(4,3,3,3,3)$ with the diagram below.  
(Since $r_v(\esse_1)=r_v(\esse_2)=r_v(\esse_3)=0$, these boxes do not move; and $r_v(\esse_4)=2$, so this box moved $2$ units NW to $\esse'_4$.)

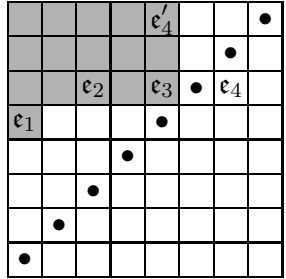
\[
\Tableau{
\gray&\gray&\gray&\gray&\gray\mathfrak{e}_4'&&&\bullet\\
\gray&\gray&\gray&\gray&\gray&&\bullet&\\
\gray&\gray&\gray\mathfrak{e}_2&\gray&\gray\mathfrak{e}_3&\bullet&\mathfrak{e}_4&\\
\gray\mathfrak{e}_1&&&&\bullet&&&\\
&&&\bullet&&&&\\
&&\bullet&&&&&\\
&\bullet&&&&&&\\
\bullet&&&&&&&\\
}
\]

The outer shape for $w\leq v$ can also be obtained from a {\it weak triple} $\triple'=(\bk',\bp,\bq)$ associated to the pair $(w,v)$.  Since $w$ is vexillary, it has a corresponding triple $\triple = (\bk,\bp,\bq)$.  Then $\triple'$ is formed by setting $k'_i = k_v(q_i,p_i)$.  (If $k'_i = k'_{i+1}$, which is possible, then the coming from $(k'_i,p_i,q_i)$ dominates; this is the stronger condition.)  Since $v\geq w$, we have $k'_i\geq k_i$ for each $i$.  The partition $\mu$ is the one associated to $\triple'$ by the previous formula: $\mu_{k'_i} = q_i-p_i+k'_i$, with other parts filled in minimally.  (In case $k'_i = k'_{i+1}$, one uses $\mu_{k'_i} = q_i-p_i+k'_i$, as this will be the larger of the two possibilities.)

\begin{remark}
It was shown in \cite{LY} that there is a unique vexillary $\Theta_{v,w}$ such that a subset of $\mathfrak{e}'$'s obtained by moving $\mathfrak{e}\in \Ess(w)$ form the essential boxes of $\Theta_{v,w}$.  In fact, this vexillary permutation can be constructed from the weak triple $\triple'$, by a procedure similar to the one constructing $w$ from $\triple$.  For our example, we have $\Theta_{v,w}=5\,4\,6\,2\,7\,1\,3\,8$.
\end{remark}

\subsection{Type C}
Let $W_n^\# = S_{2n}^\# \cap W_n$ be the set of vexillary signed permutations. For $w\in W_n^\#$, we have a subset
$\Ess^{-}(w)\subset \{\bar{n},\ldots,\bar{1},1,\ldots,n\}\times \{\bar{n},\ldots,\bar{1}\}$
such that the Schubert variety is given by
 \begin{equation}
\Omega_w=\{E_\bullet\;|\;
\dim(E_{p}\cap F_q)\geq k_w(q,p)
\;\mbox{for all}\;
(q,p)\in \Ess^-(w)
\}.\label{eq:OmegaDefEss_C}
\end{equation}
An element 
$\mathfrak{e}\in  \{\bar{n},\ldots,\bar{1},1,\ldots,n\}\times \{\bar{n},\ldots,\bar{1}\}$ 
is an \defin{essential box} of  
$w$ if 
$\mathfrak{e}$ is a south-east corner of $D^{-}(w):=D(w)\cap( \{\bar{n},\ldots,\bar{1},1,\ldots,n\}\times \{\bar{n},\ldots,\bar{1}\}) $, and not in the set 
 $\{(i,\overline{1})\;|\;i< \overline{1}\}$.
Let $\Ess^{-}(w)$ be the set of essential boxes of $w$.\footnote{For general signed permutations, the definition of essential set is slightly more complicated \cite{A18}. Here we use a simplified version which is valid in the vexillary case (see \cite[p.~8, lines~16--17]{AF18}).}

The following example shows the essential position for a signed vexillary permutation $w=\bar{1}\;\bar{2}\;3\;\bar{5}\;\bar{4}$, with diagram shown below.
\[
\Tableau{\bl&\bl\bar{5}&\bl\bar{4}&\bl\bar{3}&\bl\bar{2}&\bl\bar{1} \\
\bl\bar{5}&&&&&\\
\bl\bar{4}&&&&&\\
\bl\bar{3}&&&\gray\bullet&\gray&\gray\\
\bl\bar{2}&&&\gray&&\\
\bl\bar{1}&&&\gray&&\mathfrak{e}_3\\
\bl1&&&\gray&\mathfrak{e}_2&\gray\bullet\\
\bl2&&&\gray&\gray\bullet&\gray\\
\bl3&&\mathfrak{e}_1&\gray&\gray&\gray\\
\bl4&\gray\bullet&\gray&\gray&\gray&\gray\\
\bl5&\gray&\gray\bullet&\gray&\gray&\gray}
\]
We note that the position $(\bar{4},\bar{1})$ is not an essential box.

\subsubsection{Vexillary signed permutations}
Let $\mathfrak{e}_1,\ldots,\mathfrak{e}_s$
be the elements in $\Ess^{-}(w)$ arranged from 
south-west to north-east.  
Let us define a triple $(\pmb{k},\pmb{p},\pmb{q})$
of integer sequence associated with $w$ 
by 
\[
k_i=k_w(\mathfrak{e}_i),\quad
p_i=-(\mbox{column index of}\;\mathfrak{e}_i),\quad
q_i=(\mbox{row index of the box just below}\;\mathfrak{e}_i).
\]
Then we have $1\leq p_i,q_i\leq n$ (see \cite{A18,AF18}).
Here we adopt a different convention from the one for type A in \S 3.1.
The Schubert variety \eqref{eq:OmegaDefEss_C} associated to a vexillary signed permutation $w$ 
can be defined by
\begin{equation}
 \Omega_w = \{ E_\bullet \,|\, \dim(E_{p_i} \cap F_{q_i}) \geq k_i \text{ for }i=1,\ldots,s\}.\label{eq:Omgea_wCtriple}
\end{equation}
The triple $\tau=(\pmb{k},\pmb{p},\pmb{q})$ of integer sequences
satisfies
\begin{equation}
  \bk\colon  0 < k_1 <\cdots< k_s, \quad
  \bp\colon n\geq p_1\geq \cdots \geq p_s \geq 1, \text{\quad and\quad }\\
  \bq\colon n\geq q_1 \geq \cdots \geq q_s \geq 1,\label{eq:triple3}
\end{equation}
and
\begin{equation}
q_i+p_i+k_i>q_{i+1}+p_{i+1}+k_{i+1}
\;\mbox{for}\;1\leq i\leq s-1,\;
\mbox{and}\;
q_{s}+p_{s}+k_{s}>0.
\label{eq:triple4}
\end{equation}
One can recover a vexillary (signed) permutation $w:=w(\tau)$ in $W_n^\#$ from a triple $\tau$; see \cite{AF12,AF18}.

\subsubsection{Shape of a signed vexillary permutation}
For each vexillary signed permutation $w$, $\lambda:=\lambda(\tau)$
is the smallest strict partition such that 
\[
\lambda_{k_i}=q_i+p_i-1\quad(1\leq i\leq s).
\]
In other words, the strict partition $\lambda(\triple) = (\lambda_1 > \cdots > \lambda_{k_s} >0)$ has parts $\lambda_k = p_i+q_i-1+k_i-k$ whenever $k_{i-1}<k\leq k_i$.  (We use the convention $k_0=0$.)

A graphical construction is given as follows. If we set $r_i=r_w(\mathfrak{e}_i)$, 
the partition $\lambda=\mathrm{sh}(w)$ is obtained as follows:
We move each $\mathfrak{e}_i$ diagonally north-west by $r_i$ units,
denoted by $\mathfrak{e}_i'$. Next 
we make a Young diagram $\tilde{\lambda}$ fitting at north-west corner of the $n\times n$ grid 
having 
$\mathfrak{e}_i'$'s as corners (as in type A). 
Then we remove boxes strictly upper to the diagonal from the Young diagram  $\tilde{\lambda}$.  The desired strict partition is the transpose of this along the diagonal.

The vexillary element
$w=\bar{1}\;\bar{2}\;3\;\bar{5}\;\bar{4}$ has shape $\lambda=(8,7,3,1)$, which can be seen in Figure \ref{Shape_C}.

\begin{multicols}{2}
\centering
{\Tableau{
 \bl&\bl\bar{5}&\bl\bar{4}&\bl\bar{3}&\bl\bar{2}&\bl\bar{1} \\
\bl\bar{5}&\gray&\times&\times&\times&\times\\
\bl\bar{4}&\gray&\gray&\times&\times&\times\\
\bl\bar{3}&\gray&\gray&\gray\bullet&\times&\times\\
\bl\bar{2}&\gray&\gray&\gray&\gray\mathfrak{e}_3'&\times\\
\bl\bar{1}&\gray&\gray&\gray\mathfrak{e}_2'&&\mathfrak{e}_3\\
\bl1&\gray&\gray&&\mathfrak{e}_2&\bullet\\
\bl2&\gray&\gray&&\bullet&\\
\bl3&\gray&\gray\mathfrak{e}_1&&&\\
\bl4&\bullet&&&&\\
\bl5&&\bullet&&&}
\captionof{figure}{Shape of $w=\bar{1}\;\bar{2}\;3\;\bar{5}\;\bar{4}$}\label{Shape_C}}
{\centering
 \Tableau{
 \bl&\bl\bar{5}&\bl\bar{4}&\bl\bar{3}&\bl\bar{2}&\bl\bar{1} \\
\bl\bar{5}&\gray&\times&\times&\times&\times\\
\bl\bar{4}&\gray&\gray&\times&\times&\times\\
\bl\bar{3}&\gray&\gray&\gray&\times&\times\\
\bl\bar{2}&\gray&\gray&\gray&\gray\mathfrak{e}_3'&\times\\
\bl\bar{1}&\gray&\gray\bullet&\gray\mathfrak{e}_2'&&\mathfrak{e}_3\\
\bl1&\gray&&&\mathfrak{e}_2&\\
\bl2&\gray&&&&\bullet\\
\bl3&\gray&&&\bullet&\\
\bl4&\gray\mathfrak{e}_1&&\bullet&&\\
\bl5&\bullet&&&&}
\captionof{figure}{Outer shape}\label{outer_C}
}
\end{multicols}

\subsubsection{Outer shape}
Given a vexillary element $w$ and any $v\geq w$, the outer shape $\mu$
is the smallest shifted diagram containing all the $\mathfrak{e}'$ with $\mathfrak{e}\in \Ess(w)$.  For instance, given $w$ as above and $v=\bar{2}\;\bar{3}\;\bar{4}\;1\;\bar{5}$ in $W_5$, we move the essential boxes and take the smallest shifted diagram to get the outer shape $\mu=(9,4,3,1)$, in Figure \ref{outer_C}.

As in type A, there is a weak triple $\triple' = (\bk',\bp,\bq)$ associated to $w\leq v$, obtained by setting $k'_i = k_v(q_i,p_i)$.  The outer shape $\mu$ can be determined by setting $\mu_{k'_i} = p_i+q_i-1$, and filling in other parts minimally to obtain a strict partition.  (If $k'_i=k'_{i+1}$, we use the same convention as before, setting $\mu_{k'_i} = p_i+q_i-1$.)

\subsection{Type D}
We define $W_{n}^{+,\#}=W_{n}^+\cap S_{2n}^{\#}$, where $W_{n}^+$ is the index $2$ subgroup of $W_{n}$ consisting of signed permutations with an even number of sign changes. 

The vexillary signed permutation $w$ associated to $\tau$ in type D is the same as the vexillary permutation $w$ of type C after replacing $p_i$ by $p_i+1$ and $q_i$ by $q_{i}+1$. The Schubert variety $\Omega_w$ is described by the same conditions as in \eqref{eq:Omgea_wCtriple}.  A type D triple $\triple$ for the Schubert variety satisfies 
$$\bold{k}:0<k_1<\cdots<k_s,\;\;\bold{p}:p_1\geq\cdots\geq p_s\geq 0, \;\;\bold{q}:q_1\geq\cdots q_s\geq 0$$
of length $s$, with
\begin{equation}\label{Dtriple}
 p_i+q_i+k_i>p_{i+1}+q_{i+1}+k_{i+1}
 \end{equation}
for $1\leq i\leq s-1$ (see \cite{AF12}).  In particular, the triple $\triple$ is given by the same sequences in type C, but the last components of each sequence, that is, $p_s$ and $q_s$ can be $0$.  If $p_s=q_s=0$, then $k_s$ is required to be even.  If $k_s$ is odd, we replace the triple by one with $k_{s+1}=k_s+1$ and $p_{s+1}=q_{s+1}=0$.

\subsubsection{Shape of a signed vexillary permutation}
Given a vexillary signed permutation $w$ with a triple $\tau$, the smallest strict partition $\lambda:=\lambda(\tau)$ is defined by
$$\lambda_{k_i}=p_i+q_i \quad (1\leq i\leq s),$$
and $\lambda_k=\lambda_{k_i}+k_i-k$ for $k_{i-1}< k\leq k_i$ with the convention $k_0=0$.  Note that $\lambda_{k_s}=0$ if $p_s=q_s=0$.  Also if we set the strict partition $\lambda^{+}=\lambda(\tau^{+})$ defined by $\lambda^{+}=(\lambda_1+1,\ldots,\lambda_s+1),$ $\lambda^{+}$ becomes the strict partition for type C.

The shape of a vexillary permutation can also be seen from its diagram.  For example, let us take $w=3\;\bar{1}\;\bar{2}\;4\;5$. The type C shape is $\lambda^+=(4,2)$.  Then we remove the diagonal boxes from the shape $\lambda^+$ to obtain the shape $\lambda=(3,1)$ of type D, as in Figure \ref{Shape_B}. 

\begin{multicols}{2}
\centering
\Tableau{\bl&\bl\bar{5}&\bl\bar{4}&\bl\bar{3}&\bl\bar{2}&\bl\bar{1} \\
\bl\bar{5}&\bullet\times&\times&\times&\times&\times\\
\bl\bar{4}&\gray&\bullet\times&\times&\times&\times\\
\bl\bar{3}&\gray&\gray\mathfrak{e}_2'&\times&\times&\bullet\times\\
\bl\bar{2}&\gray\mathfrak{e}_1'&&&\times&\times\\
\bl\bar{1}&&&&\mathfrak{e}_2&\times\\
\bl1&&&\mathfrak{e}_1&\bullet&\\
\bl2&&&\bullet&&\\
\bl3&&&&&\\
\bl4&&&&&\\
\bl5&&&&&\\}
\captionof{figure}{Shape of $w=3\;\bar{1}\;\bar{2}\;4\;5$}\label{Shape_B}
{\centering
 \Tableau{\bl&\bl\bar{5}&\bl\bar{4}&\bl\bar{3}&\bl\bar{2}&\bl\bar{1} \\
\bl\bar{5}&\times&\times&\times&\times&\times\\
\bl\bar{4}&\gray&\times&\times&\times&\times\\
\bl\bar{3}&\gray&\gray&\times&\times&\bullet\times\\
\bl\bar{2}&\gray&\gray&\gray&\times&\times\\
\bl\bar{1}&\gray&\gray\mathfrak{e}_1'&\gray&\gray\mathfrak{e}_2&\times\\
\bl1&&\bullet&\mathfrak{e}_1&&\\
\bl2&&&&\bullet&\\
\bl3&&&&&\\
\bl4&&&\bullet&&\\
\bl5&\bullet&&&&\\}\\
\captionof{figure}{Outer shape}\label{outer_B}
}
\end{multicols}

\subsubsection{Outer shape}
Given a type D vexillary permutation $w$ with triple $\triple$, and $v\geq w$, the associated outer shape $\mu$ can be defined as before from the weak triple $\triple'$.  Alternatively, $\mu^+ = (\mu_1+1,\ldots,\mu_{k_s}+1)$ is the type C outer shape for the pair $w\leq v$. 

For example, consider
$w=3\;\bar{1}\;\bar{2}\;4\;5$ and $v= 3\;\bar{2}\;\bar{4}\;\bar{1}\;\bar{5}$, so $w$ is vexillary (as before) and $v\geq w$.  Then from Figure \ref{outer_B}, we get $\mu^+=(5,4,3,2)$. By removing the diagonal parts, we have the outer shape $\mu=(4,3,2,1)$.

\subsection{Type B}
This is essentially the same as type C.  The Weyl group is $W=W_n$.  
We use the same parametrization of vexillary signed permutations by (type C) triples, and the description of $\Omega_w$ in terms of the triple looks the same as \eqref{eq:Omgea_wCtriple}.  The shape $\lambda(\triple)$ is the same as in type C, as is the outer shape $\mu$.

\section{The multiplicity formula}

Now we turn to our main theorem.  First, we recall the definition and basic properties of the multiplicity we are computing.

Let $X$ be an algebraic variety containing a point $p$.  Let $R=\mathcal{O}_{X,p}$ be the local ring of $X$ at $p$ with, maximal ideal $\mathfrak{m}$.  The \emph{Hilbert-Samuel polynomial} of $R$ is given by 
 $$\mathcal{P}_R(n)=\text{dim}_{\KK}(R/\mathfrak{m}^n)=(u/d!)\;n^d+\cdots$$  for sufficiently large $n$, where $d$ is the Krull dimension of $R$ and $u$ is a nonnegative integer.  The \emph{Hilbert-Samuel multiplicity} of $R$ is the leading coefficient: it is defined as
 $$\mathrm{mult}_p(X)=u.$$ 
The following properties hold:
\begin{enumerate}
\item If $U\subset X$ is a (Zariski or \'etale) neighborhood of $p$, then $\mult_p(X)=\mult_p(U)$.

\item If $(X,p) \xrightarrow{\sim} (X',p')$ is an isomorphism, then $\mult_p(X) = \mult_{p'}(X')$.

\item For any affine space $\mathbb{A}^n$ with origin $0$, we have $\mult_p(X) = \mult_{(p,0)}(X\times\mathbb{A}^n)$.
\end{enumerate}
These properties may be expressed concisely as follows: if $X' \to X$ is a smooth morphism, sending $p' \mapsto p$, then $\mult_{p'}(X') = \mult_p(X)$.

Next, we review the notion of {\it excited Young diagram}, following \cite{IN}.  Consider a pair (ordinary or shifted) Young diagrams $\lambda \subset \mu$.  An {\it excitation} of $\lambda$ is a collection of boxes inside $\mu$ which are obtained from $\lambda$ by a sequence of {\it elementary excitations}.  These depend on type.  In type A, an elementary excitation is a local move of the form
\[
\Tableau{\gray&\\&}\leadsto\Tableau{&\\&\gray}\;.
\]
In type C, an elementary excitation is one of the following:
\[
\Tableau{\gray&\\~&}\leadsto\Tableau{&\\~&\gray}\;,\qquad
\Tableau{\gray&\\&}\leadsto\Tableau{&\\&\gray}\;.\\
\]
\vspace{0.2cm}
In types B and D, elementary excitations are of the form
\begin{equation}
\Tableau{\gray&\\&}\leadsto\Tableau{&\\&\gray}\;,\qquad\Tableau{\gray& &~\\~&&\\~&~&}\leadsto\Tableau{&&~\\~&&\\~&~&\gray}\;.\label{EQ}
\end{equation}
In each case, we write $\EYD_\mu(\lambda)$ for the set of excitations of $\lambda$ inside $\mu$, the type being understood from context.

For Schubert varieties in Grassmannians (ordinary, Lagrangian, or maximal orthogonal), we have the following formula for the multiplicity, due in this form to Ikeda-Naruse \cite[\S9]{IN}:
\begin{equation}\label{e.GrMult}
 \mult_{\fp_\mu}\Omega_\lambda = \#\EYD_\mu(\lambda).
\end{equation}

Now we can state our main theorem.  Let $G$ be a classical group, so one of $SL_n$, $Sp_{2n}$, $SO_{2n+1}$, or $SO_{2n}$, with Weyl group $W$.  We consider Schubert varieties in the corresponding flag variety $G/B$ (which is one of $Fl(V)$, $Fl^C(V)$, $Fl^B(V)$, or $Fl^D(V)$, for an appropriate vector space $V$).

\begin{theorem}\label{MainTh}
Let $w$ and $v$ be elements of $W$, with $w$ vexillary and $v\geq w$. Let $\lambda=\mathrm{sh}(w)$ be the shape of $w$, and let $\mu$ be the outer shape associated to $w$ and $v$.  Then the Hilbert-Samuel multiplicity of $\Omega_w$ at $\fp_v$ is given by the formula
\[
  \mult_{\fp_v}(\Omega_w)=\#\mathcal{E}_{\mu}(\lambda).
\]
\end{theorem}

We illustrate the theorem with examples in each type.

\begin{example}[Type A] 
Consider permutations $v=87654321\geq w=54617283$.  The shape of $w$ is $\lambda = (4,3,3,2,1)$ and the outer shape is $\mu = (4,3,3,3,3)$.  Then we have $\mult_{\fp_v}(\Omega_w)=2$, computed from the excited Young diagrams shown below.
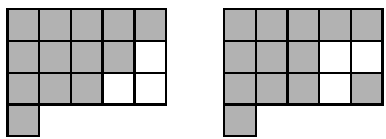
\begin{align*}
\small\Tableau{
\gray&\gray&\gray&\gray&\gray\\
\gray&\gray&\gray&\gray&\\
\gray&\gray&\gray&&\\
\gray\\
}&\quad\quad
\small\Tableau{
\gray&\gray&\gray&\gray&\gray\\
\gray&\gray&\gray&&\\
\gray&\gray&\gray&&\gray\\
\gray
}
\end{align*}
More generally, in type A there are bijections between several combinatorial objects: flagged set-valued tableaux, pipe dreams, and excited Young diagrams (see \cite{LY}).  These sets are also enumerated by certain binomial determinants \cite{IN,Kra,LY}.
\end{example}

\begin{example}[Type C] 
Take $w=\bar{1}\;\bar{2}\;3\;4,\; v=\bar{2}\;\bar{3}\;\bar{4}\;1$.  We know that $w\leq v$ such that $\lambda=\mathrm{sh}(w)=(3,1)$ and $\mu=(4,3,1)$.  By the type C elementary excitations, we have $\mult_{\fp_v}(\Omega_w)=6$.

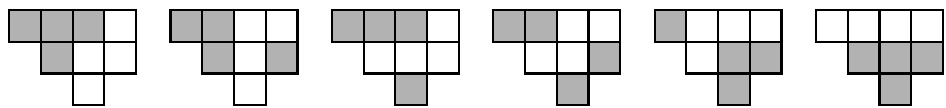
$$
\small\Tableau{
\gray&\gray&\gray&\\
~&\gray&&
\\~&~&}\;\quad
\small\Tableau{
\gray&\gray&&\\
~&\gray&&\gray
\\~&~&}\;\quad
\small\Tableau{
\gray&\gray&\gray&\\
~&&&
\\~&~&\gray}\;\quad
\small\Tableau{
\gray&\gray&&\\
~&&&\gray
\\~&~&\gray}\;\quad
\small\Tableau{
\gray&&&\\
~&&\gray&\gray
\\~&~&\gray}\;\quad
\small\Tableau{
&&&\\
~&\gray&\gray&\gray
\\~&~&\gray}\;
$$

\end{example}

\begin{example}[Type D]  Let $w=3\;\bar{1}\;\bar{2}\;4\;5$ and $v= 3\;\bar{2}\;\bar{4}\;\bar{1}\;\bar{5}$. We know $w\leq v$ in Bruhat order, so that $\fp_v\in \Omega_w$. Since $\lambda=(3,1)$ and $\mu=(4,3,2,1)$, by applying the theorem, we get $\mathrm{mult}_{\fp_v}(\Omega_w)=5$ with the following excited states. 

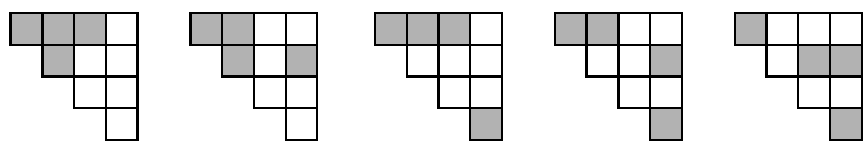
\begin{align*}
\small\Tableau{\gray&\gray&\gray&\\
~&\gray&&\\
~&~&&\\
~&~&~&}\quad\quad
\small\Tableau{\gray&\gray&&\\
~&\gray&&\gray\\
~&~&&
\\~&~&~&}&\quad\quad
\small\Tableau{\gray&\gray&\gray&\\
~&&&\\
~&~&&
\\~&~&~&\gray}
\quad\quad
\small\Tableau{\gray&\gray&&\\
~&&&\gray\\
~&~&&
\\~&~&~&\gray}\quad\quad
\small\Tableau{\gray&&&\\
~&&\gray&\gray\\
~&~&&\\
~&~&~&\gray}
\end{align*}
\end{example}

\begin{example}[Type B] 
Let us consider a vexillary permutation $w=\overline{1}\;\overline{2}\;3\;4$ with a fixed point $\fp_v$ in $\Omega_w$ for $v=\overline{2}\;\overline{3}\;\overline{4}\;1$. Since the inner shape is $(3,1)$ and the outer shape is $(4,3,1)$, by using Theorem \ref{MainTh} with type B excited Young diagrams, we get $\text{mult}_{\fp_v}(\Omega_w)=2$.

$$
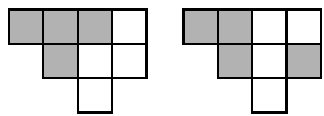
\Tableau{
\gray&\gray&\gray&\\
~&\gray&&
\\~&~&}\;\quad
\Tableau{
\gray&\gray&&\\
~&\gray&&\gray
\\~&~&}$$

\end{example}

To prove Theorem~\ref{MainTh}, we reduce to the Grassmannian case, where the formula is given by \eqref{e.GrMult}.  The details occupy the remainder of the paper; here we give a brief outline.

In each type, the vexillary Schubert variety $\Omega_w \subset Fl(V)$ is the preimage of an analogously defined Schubert variety in a partial flag variety, $\Omega_{[w]} \subset Fl(\bp,V)$ (see \S\ref{s.dirsum} for the definition).  Using smooth invariance of multiplicities, we work with Schubert varieties in these partial flag varieties.  Here, we employ a direct sum embedding to map a vexillary Schubert variety $\Omega_w$ to an inverse-Grassmannian Schubert variety $\Omega_{w_\lambda^{-1}}$ in a larger (partial) flag variety.  This embedding sends $\fp_v$ to $\fp_{w_\mu^{-1}}$, up to an action by a group which preserves $\Omega_{w_\lambda^{-1}}$.

The main technical step is an isomorphism (up to product with affine space) between local neighborhoods of $\fp_v$ in $\Omega_w$ and of $\fp_{w_\mu^{-1}}$ in $\Omega_{w_\lambda^{-1}}$.  More precisely, we show the corresponding Kazhdan-Lusztig varieties, $X_v^\circ \cap \Omega_w$ and $X_{w_\mu^{-1}}^\circ \cap \Omega_{w_\lambda^{-1}}$, are isomorphic up to a product with affine space (Theorem~\ref{t.KLisom}).

At this point we have demonstrated $\mult_{\fp_v}\Omega_w = \mult_{\fp_{w_\mu^{-1}}}\Omega_{w_\lambda^{-1}}$.  To finish, we apply a general local isomorphism---valid for all Schubert varieties in any $G/B$---between $(\Omega_w, \fp_v)$ and $(\Omega_{w^{-1}},\fp_{v^{-1}})$, to conclude
\[
 \mult_{\fp_{w_\mu^{-1}}}\Omega_{w_\lambda^{-1}} = \mult_{\fp_{w_\mu}}\Omega_{w_\lambda} = \mult_{\fp_{\mu}}\Omega_{\lambda},
\]
using the smooth invariance of multiplicities (and the projection to the Grassmannian) in the last equality.  The theorem then follows from \eqref{e.GrMult}.

\section{A local isomorphism}\label{s.locisom}

We need a lemma which relates $w$ to $w^{-1}$.

\begin{lemma}\label{l.inverse-isom}
Let $X_w \subset G/B$ be a $B$-invariant (opposite) Schubert variety, with $\fp_v\in X_w$ a fixed point corresponding to $v\leq w$.  Then the local ring of $X_w$ at $\fp_v$ is isomorphic to the local ring of $X_{w^{-1}}$ at $\fp_{v^{-1}}$.
\end{lemma}

\begin{proof}
Consider the subvariety $Z(w) \subseteq G/B \times G/B$ defined by
\[
  Z(w) = \overline{G\cdot (\fp_{e}, \fp_w)}.
\]
Let $pr_1$ and $pr_2$ be the projections $Z(w) \to G/B$.  Then $pr_1^{-1}(\fp_{e}) = X_w$, while $pr_2^{-1}(\fp_{e}) = X_{w^{-1}}$.  Both of these are locally trivial fiber bundles, so there is a neighborhood $U_e \subset G/B$ such that $pr_1^{-1}(U_e) \isom U_e \times X_w$, and $pr_2^{-1}(U_e) \isom X_{w^{-1}} \times U_e$.  So, up to a product with affine space, we have local isomorphisms of $Z(w)$ at $(\fp_e,\fp_v)$ with $X_w$ at $\fp_v$, and of $Z(w)$ at $(\fp_{v^{-1}},\fp_e)$ with $X_{w^{-1}}$ at $\fp_{v^{-1}}$.  Multiplication by (a coset representative for) $v^{-1}$, diagonally on $G/B \times G/B$, defines an automorphism of $Z(w)$ which sends $(\fp_e,\fp_v)$ to $(\fp_{v^{-1}},\fp_e)$.  Composing these isomorphisms proves the lemma.
\end{proof}

\begin{remark}
It is not true, in general, that the Schubert varieties $X_w$ and $X_{w^{-1}}$ are (globally) isomorphic.  See, e.g., \cite{RS} for counterexamples and a criterion for global isomorphism. 
\end{remark}

\begin{remark}
Let $\Omega_w = \overline{B_-\cdot \fp_w} = w_\circ\cdot X_{w_\circ w}$, where $w_\circ$ is the longest element of $W$.  A direct translation of the lemma establishes a local isomorphism of $(\Omega_w,\fp_v)$ with $(\Omega_{w_\circ w^{-1} w_\circ}, \fp_{w_\circ v^{-1} w_\circ})$.  But using an isomorphism $G/B \xrightarrow{\sim} G/B_-$, one can identify $\Omega_w$ with $\Omega_{w_\circ w w_\circ}$, sending $\fp_v$ to $\fp_{w_\circ v w_\circ}$.  Applying this, we obtain a local isomorphism of $(\Omega_w,\fp_v)$ with $(\Omega_{w^{-1}},\fp_{v^{-1}})$.
\end{remark}

\section{The direct sum embedding}\label{s.dirsum}

Here we describe the direct sum embedding $\Sigma$, which we will use repeatedly in what follows.  For us, the key property is that a vexillary Schubert variety is the transverse intersection of the image of $\Sigma$ with an inverse-Grassmannian Schubert variety in the (larger) target flag variety.  The construction is similar in each type; while we spell out the details in each case, the reader is encouraged to focus on the type A case, which contains all the necessary information.

First, we make a general remark.  Let $P$ be a parabolic subgroup containing the Borel subgroup $B$ of $G$.  The projection $\pi\colon G/B \to G/P$ is a locally trivial fiber bundle, with smooth fibers isomorphic to $P/B$.  Schubert varieties and fixed points of $G/P$ are indexed by cosets $W/W_P$, where $W_P\subset W$ is the Weyl group of $P$.  The projection sends a point $wB$ to $\pi(wB) = [w]P$, where the coset $[w]$ is in $W/W_P$.  Any $w \in W$ can be uniquely decomposed into $w^{min}w_P$ for some $w_P\in W_P$, where $w^{min}$ is the minimal representative for $[w]$.  With this notation, we have $\pi^{-1}\Omega_{[w]} = \Omega_{w^{min}}$, so the restriction $\Omega_w\rightarrow \Omega_{[w]}$ is also a fiber bundle with smooth fibers.  In particular, the multiplicity of a point $p_{[v]} \in \Omega_{[w]}$ is the same as that of $p_{v} \in \Omega_{w^{min}}$.

\subsection*{Type A}

Let $w$ be a vexillary permutation in $S_n^\#$, with triple $(\bk,\bp,\bq)$.  The partial fixed flag $F^{q_1} \subset \cdots \subset F^{q_s} \subset V$ suffices to define the Schubert variety $\Omega_w$.  We also consider the partial flag variety $Fl(\bp;V)$ parametrizing flags $E_{p_1} \subset \cdots \subset E_{p_s} \subset V$, with the projection $\pi\colon Fl(V) \to Fl(\bp;V)$.  By the main fact about essential sets---and as noted in the previous paragraph---we have $\Omega_w = \pi^{-1}\Omega_{[w]}$, where $\Omega_{[w]} \subset Fl(\bp;V)$ is the subvariety defined by the same conditions, $\dim(E_{p_i} \cap F^{q_i}) \geq k_i$ for $1\leq i\leq s$.  This is because, by construction, a vexillary permutation $w$ coming from $(\bk,\bp,\bq)$ is minimal in its coset for the projection to $Fl(\bp;V)$.

Since $\pi$ is a smooth morphism, this reduces the study of singularities of $\Omega_w$ to those of $\Omega_{[w]}$.  (The analogous statements hold, for the same reasons, in other types.)

Given $\bp$ and $\bq$, let $r_i = p_i+n-q_i$, and write $\br = (r_1,\ldots,r_s)$.   
We have an embedding
\begin{equation}\label{e.typeAsum}
 \Sigma \colon Fl(\bp;V) \hookrightarrow Fl(\br; V\oplus V)
\end{equation}
defined by sending
\[
  E_{p_1} \subset \cdots \subset E_{p_s} \subset V \quad \text{to} \quad  E_{p_1}\oplus F^{q_1} \subset \cdots \subset E_{p_s} \oplus F^{q_s} \subset V \oplus V.
\]
We call this the {\it direct sum embedding}.

Let $\Delta = \{(v,v) \,|\, v\in V \} \subset V \oplus V$ be the diagonal subspace, and consider the locus
\[
  \Omega_{w_\lambda^{-1}} = \{ G_\bullet \, |\, \dim( G_{r_i} \cap \Delta ) \geq k_i \text{ for } 1\leq i\leq s \}
\]
in $Fl(V\oplus V)$.  The notation is intentional: using any fixed flag in $V\oplus V$ which contains $\Delta$ as its $n$-dimensional component, this locus is the Schubert variety associated to the inverse of the  Grassmannian permutation $w_\lambda \in S_{2n}$ of descent $n$, where $\lambda$ is the shape of $w$, so $\lambda_{k_i} = q_i-p_i+k_i$.  The Schubert variety $\Omega_{[w_\lambda^{-1}]} \subset Fl(\br; V\oplus V)$ is defined by the same conditions.  So $\Omega_{w_\lambda^{-1}} = \Pi^{-1}\Omega_{[w_\lambda^{-1}]}$, where $\Pi\colon Fl(V\oplus V) \to Fl(\br; V\oplus V)$ is the projection.

The key fact about the direct sum embedding is this: for the vexillary permutation $w$ with triple $(\bk,\bp,\bq)$ and shape $\lambda$,
\begin{equation}
  \Omega_{[w]} = \Sigma^{-1}\Omega_{[w_\lambda^{-1}]}
\end{equation}
in $Fl(\bp;V)$.  All this can be summarized by the diagrams of fiber squares:
\[
\begin{tikzcd}
  Fl(\bp;V) \ar[r,"\Sigma",hook]  & Fl(\br; V\oplus V)  \\
  \Omega_{[w]} \ar[u,hook] \ar[r,hook] & \Omega_{[w_\lambda^{-1}]} \ar[u,hook]
\end{tikzcd}
\]
\[
\begin{tikzcd}
  \Omega_w \ar[r,hook] \ar[d]  & Fl(V) \ar[d,"\pi"]  \\
  \Omega_{[w]}  \ar[r,hook] &  Fl(\bp;V)
\end{tikzcd}
\qquad
\begin{tikzcd}
  \Omega_{w_\lambda^{-1}} \ar[r,hook] \ar[d]  & Fl(V\oplus V) \ar[d,"\Pi"]  \\
  \Omega_{[w_\lambda^{-1}]}  \ar[r,hook] &  Fl(\br;V\oplus V).
\end{tikzcd}
\]

\subsection*{Type C}

Let $w$ be a vexillary signed permutation in $W_n^\#$, with (type C) triple $(\bk,\bp,\bq)$.  We use the notation $Fl^C(\bp,V)$ to denote the partial isotropic flag variety of subspaces $E_{p_1} \subset \cdots \subset E_{p_s} \subset V$, with $\dim E_{p_i} = n+1-p_i$, all isotropic with respect to the symplectic form $\omega$.  Similarly, we have the fixed isotropic partial flag $F_{q_1} \subset \cdots \subset F_{q_s} \subset V$.

On the vector space $V\oplus V$ there is a canonical symplectic form $\langle\!\langle\;,\;\rangle\!\rangle$, defined by 
\[
  \langle\!\langle v_1\oplus v_2, w_1\oplus w_2 \rangle\!\rangle = \langle v_1,w_1 \rangle - \langle v_2,w_2 \rangle,
\]
where $\langle\;,\;\rangle$ is the given symplectic form on $V$.  
This has the property that $E\oplus F \subset V\oplus V$ is isotropic whenever both $E,F \subset V$ are isotropic, and also that the diagonal subspace $\Delta\subset V\oplus V$ is isotropic.

The direct sum embedding is defined as before:
\[
 \Sigma\colon Fl^C(\bp;V) \hookrightarrow Fl^C(\br;V\oplus V)
\]
sends
\[
  E_{p_1} \subset \cdots \subset E_{p_s} \subset V \quad \text{to} \quad  E_{p_1}\oplus F_{q_1} \subset \cdots \subset E_{p_s} \oplus F_{q_s} \subset V \oplus V.
\]
Here $r_i = p_i+q_i-1$, and in accordance with our type C conventions, the subspace $G_{r_i} \subseteq V\oplus V$ has dimension $2n+1-r_i = (n+1-p_i)+(n+1-q_i)$.

Also as before, for a strict partition $\lambda$, we have the locus
\begin{equation}\label{locusC}
  \Omega_{w_\lambda^{-1}} = \{ G_\bullet \, |\, \dim( G_{r_i} \cap \Delta ) \geq k_i \text{ for } 1\leq i\leq s \},
\end{equation}
an inverse-Grassmannian Schubert variety in $Fl^C(V\oplus V)$, 
and
\begin{equation}
  \Omega_{[w]} = \Sigma^{-1}\Omega_{[w_\lambda^{-1}]}.
\end{equation}

The type C direct sum embedding is compatible with the type A one, requiring only notational changes.  That is, having fixed our isotropic flag $F_\bullet$, the diagram
\begin{equation}\label{e.dirsum-commute}
\begin{tikzcd}
  Fl(\tilde\bp;V) \ar[r,"\Sigma",hook]  & Fl(\tilde\br; V\oplus V)  \\
   Fl^C(\bp;V)  \ar[u,hook] \ar[r,"\Sigma",hook] & Fl^C(\br; V\oplus V)  \ar[u,hook]
\end{tikzcd}
\end{equation}
commutes, where $\tilde{p}_i = n+1-p_i$ and $\tilde{r}_i = 2n+1-r_i$ (so these index the dimensions of the subspaces parametrized by the partial flags).

Later it will be useful to embed isotropic flag varieties in type A flag varieties by remembering the {\it coisotropic} spaces as well, so $Fl^C(\bp;V) \hookrightarrow Fl(\tilde\bp;V)$ sends $E_{p_1} \subset \cdots \subset E_{p_s}$ to the flag $E_{p_1} \subset \cdots \subset E_{p_s} \subset E_{p_s}^\perp \subset \cdots \subset E_{p_1}^\perp$.  In this case, $\tilde{\bp} = (\tilde{p}_1,\ldots,\tilde{p}_s,\tilde{p}_{s+1},\ldots,\tilde{p}_{2s})$, where  $\tilde{p}_i = n+1-p_i$ for $1\leq i\leq s$, and  $\tilde{p}_i = n-1+p_{2s+1-i}$ for $s+1\leq i\leq 2s$.  Also extending the fixed flag $F_\bullet$ to include coisotropic spaces, this embedding is compatible with direct sum in the same way, as indicated by diagram \eqref{e.dirsum-commute}.

\subsection*{Type D}

The construction is exactly the same as type C, using the symmetric form $\langle\!\langle\;,\;\rangle\!\rangle$ on $V\oplus V$ defined by
\[
  \langle\!\langle v_1\oplus v_2, w_1\oplus w_2 \rangle\!\rangle = \langle v_1,w_1 \rangle - \langle v_2,w_2 \rangle,
\]
where $\langle\;,\;\rangle$ is the given symmetric form on $V$.

The direct sum embedding $\Sigma\colon Fl^D(\bp;V) \hookrightarrow Fl^D(\br;V\oplus V)$ works as for type C, where this time $r_i = p_i+q_i$.  We have the locus 
$\Omega_{[w_\lambda^{-1}]}$ in $Fl^D(\br;V\oplus V)$ as in \eqref{locusC}, with $\Omega_{[w]} = \Sigma^{-1}\Omega_{[w_\lambda^{-1}]}$.  And this embedding is compatible with the type A one, just as in the diagram \eqref{e.dirsum-commute}.

\subsection*{Type B}

Given an odd-dimensional vector space $V$ with symmetric bilinear form, and a vexillary permutation $w\in W_n^\#$ with (type B) triple $(\bk,\bp,\bq)$, the direct sum map is
\[
 \Sigma\colon Fl^B(\bp;V) \to Fl^D(\br;V\oplus V),
\]
where the symmetric form $\langle\!\langle\;,\;\rangle\!\rangle$ on $V\oplus V$ is defined as before.  Here $r_i = p_i+q_i-1$.  Note that, in contrast to types C and D, this takes a type B flag variety to a type D one. Similarly, we have the locus $\Omega_{[w_\lambda^{-1}]}$ inside $Fl^D(\br;V\oplus V)$, defined by the same conditions as in type C \eqref{locusC}.

\section{Isomorphisms of Kazhdan-Lusztig varieties}\label{s.KLisom}

Recall that $X_v^\circ$ denotes an opposite Schubert cell, the $B$-orbit of a fixed point $\fp_v$, so it is an affine space of dimension $\ell(v)$, and $\Omega_w$ is a Schubert variety, of codimension $\ell(w)$.

For any $v\in W$, an affine neighborhood of $\fp_v$ is given by $v\,\Omega_{\mathrm{id}}^\circ$.  To study the Schubert variety $\Omega_w$ locally at the
point $\fp_v$ ($w\leq v$), we only need to understand the affine variety $\Omega_w\cap v\,\Omega_{\mathrm{id}}^\circ$. 
However, as observed by Kazhdan and Lusztig \cite[Lemma A.4]{KL79}, 
there is an isomorphism
\[
\Omega_w\cap v\,\Omega_{\mathrm{id}}^\circ \cong
(\Omega_w\cap X_v^\circ)\times\mathbb{A}^{\dim (G/B)-\ell(v)}.
\]
So we study the affine variety $\Omega_w\cap X_v^\circ$, often called a {\it Kazhdan-Lusztig variety\/}.

In our setting, $w=w(\triple)$ is the vexillary (signed) permutation associated to a triple $\triple = (\bk,\bp,\bq)$, and $v\geq w$.  Recall that we defined a sequence $\bk' = (k_1' < \cdots < k_s')$ by setting $k'_i = k_v(p_i,q_i)$ for each $i$, obtaining a weak triple $\triple' = (\bk',\bp,\bq)$.  We have partitions $\lambda$ and $\mu$ associated to $\triple$ and $\triple'$, respectively; $\lambda$ is the shape of $w$, and $\mu$ is the outer shape of the pair $v\geq w$.  These partitions, in turn, have associated Grassmannian (signed) permutations $w_\lambda$ and $w_\mu$.

\begin{theorem}\label{t.KLisom}
With notation as above, so $w$ is vexillary and $v\geq w$, with corresponding partitions $\mu \supseteq \lambda$, we have an isomorphism
\[
  \Omega_w \cap X_v^\circ \isom (\Omega_{w_\lambda^{-1}} \cap X_{w_\mu^{-1}}^\circ) \times \mathbb{A}^{\ell(v)-|\mu|}.
\]
\end{theorem}

\begin{proof}
We begin the proof with some reductions.  First, consider the projection to the partial flag variety $\pi\colon G/B \to G/P_{\bp}$, where $G/P_{\bp}$ parametrizes partial flags $E_{p_1} \subseteq \cdots \subseteq E_{p_s} \subseteq V$.  (For example, in type A, $P_{\bp}$ is block-upper-triangular, with blocks of size $p_1$, $p_2-p_1$, etc.)  We have $\Omega_w = \pi^{-1}\Omega_{[w]}$, and the map $X_v^\circ \to X_{[v]}^\circ$ identifies $X_v^\circ \isom X_{[v]}^\circ \times \mathbb{A}^{\ell(v)-\ell(v^{\min})}$, where $v^{\min}$ is the minimal representative of the coset $[v]$.  This shows
\[
  \Omega_w \cap X_v^\circ \isom (\Omega_{[w]} \cap X_{[v]}^\circ) \times \mathbb{A}^{\ell(v)-\ell(v^{\min})}.
\]
So from now on we may assume $v=v^{\min}$, i.e., $v$ is minimal in its coset with respect to $P_{\bp}$.  

Next we turn to the direct sum embedding.  Let us write $\bar{G}/\bar{P}_{\br}$ for the target of $\Sigma$.  (So in type A, $\bar{G}=GL_{2n}$ and $\bar{G}/\bar{P}_{\br}$ is $Fl(\br; V\oplus V)$, where $r_i = p_i+n-q_i$.) 
With our assumption that $v=v^{\min}$, the composition
\[
  X_v^\circ \xrightarrow{\pi} X_{[v]}^\circ \xrightarrow{\Sigma} \bar{G}/\bar{P}_{\br}
\]
is an embedding, since the projection is an isomorphism.

As in Section~\ref{s.dirsum}, let $\Omega_{w_\lambda^{-1}} \subset \bar{G}/\bar{P}_{\br}$ be defined by intersection with the diagonal subspace $\Delta \subset V\oplus V$.   And let $Q \subset \bar{G}$ be the parabolic subgroup preserving $\Delta$, so $Q$ acts on $\Omega_{w_\lambda^{-1}}$.  (Choosing an appropriate basis for $V\oplus V$, this subgroup, $Q$ is given by block lower-triangular matrices in $\bar{G}$.)

Since $\Sigma^{-1}\Omega_{w_\lambda^{-1}} = \Omega_w$, we have
\[
  X_v^\circ \cap\Omega_w \cong \Sigma(X_v^\circ) \cap\Omega_{w_\lambda^{-1}}. 
\]
Then by Lemma \ref{l.KLisom} below, the statement follows. In fact, because $Q$ preserves $\Omega_{w_\lambda^{-1}}$, and the unipotent subgroup $U$ is isomorphic to affine space, the map of the lemma induces an isomorphism
\[
\Sigma(X_v^\circ) \cap\Omega_{w_\lambda^{-1}}\cong (X_{w_\mu^{-1}}^\circ \cap \Omega_{w_\lambda^{-1}}) \times \mathbb{A}^{\ell(v)-|\mu|},
\]
as required.
\end{proof}

We continue to assume $v=v^{\min}$.
\begin{lemma}\label{l.KLisom}
There is an element $g\in Q$ and a unipotent subgroup $U \subset Q$ such that the multiplication map defines an isomorphism of affine spaces
\begin{align}\label{Iso1}
U \times  X_{w_\mu^{-1}}^\circ  &\xrightarrow{\sim} \Sigma(X_v^\circ), \\
( u, x ) &\mapsto u\cdot g \cdot x. \nonumber
\end{align}
\end{lemma}

To prove the lemma, we carry out computations in matrices: the argument consists of keeping track of reduction to row echelon form.  Nearly all the essential details appear in type A, so we will describe that case carefully, indicating what changes for other types.

\subsection{Type A}

Recall that $\lambda_{k_i} = q_i-p_i+k_i$, with the other parts filled in minimally, and similarly $\mu_{k'_i} = q_i-p_i+k'_i$.  We set $r_i = p_i+n-q_i$.  The target of $\Sigma$ is the partial flag variety $Fl(\br; V\oplus V) = \bar{G}/\bar{P}_{\br}$, where $\bar{G}=GL_{2n}$, and $\bar{P}_{\br}$ is the block-upper-triangular parabolic subgroup stabilizing a flag.  
To establish the isomorphism \eqref{Iso1}, we first represent $\Sigma(X_v^\circ)$ by $2n\times 2n$ matrices.

Choose a basis $\e_1,\ldots,\e_n$ for $V$, such that $F^{q_i}$ is the span of $\e_n,\ldots, \e_{q_i+1}$.  Start with the standard $n\times n$ matrix representatives for $X_v^\circ$.  This is the $B$-orbit of the permutation matrix for $v$, so matrices for $X_v^\circ$ have $1$'s in positions $(v(j),j)$, free entries (to be written as $*$'s) in positions $(i,j)$ such that $v(j)>i$ and $v^{-1}(i)>j$, and $0$'s elsewhere.  Say the columns of this $n\times n$ matrix are $\pmb{c}_1,\ldots,\pmb{c}_n$.

In the course of the proof, we will use a labelling of the free entries in such matrices, coming from the positions $(q_i,p_i)$ as follows.  Start with the northwest submatrix whose southeast corner lies at $(q_1,p_1)$.  Assign the label `$1$' to each $*$ in this submatrix not having a pivot $1$ in its row or column, within this same submatrix.  Continue for each $i$ from $1$ to $s$: assign label `$i$' to each previously unlabelled $*$ in the northwest $q_i\times p_i$ submatrix, if there is no pivot $1$ lying in the same row or column within the submatrix.  Some entries may be left unlabelled.

For a running example, let us take $\triple = (\bk,\bp,\bq) = (1\,2\,3\,5\,7,\, 3\,5\,5\,6\,7,\, 6\,6\,4\,2\,1)$, so $w=1\,3\,7\,5\,8\,4\,2\,6$.  And $v = 6\,7\,9\, 3\,8\,4\,5\,1\,2$ has $\bk' = (2\,3\,4\,6\,7)$.  So
\[
  \tau'=(\bk',\bp,\bq) = (2\;3\;4\;6\,7,\, 3\;5\;5\;6\,7,\, 6\;6\;4\;2\,1),
\]
and $\mu = (5,5,4,3,2,2,1)$.  In $Fl(\bp;V)$, we have:
\[ 
X_v^\circ=
\left[\begin{array}{ccc|cc|c|c|cc}
{\color{blue}*_3}&{\color{blue}*_1}&{\color{blue}*_1}&{\color{blue}*_4}&{\color{blue}*_2}&{\color{blue}*_4}&{\color{blue}*_5}_{\color{red}\lrcorner}& {\color{blue}\bf 1}&0\\
{\color{blue}*_3}&{\color{blue}*_1}&{\color{blue}*_1}&{\color{blue}*_4}&{\color{blue}*_2}&{\color{blue}*_4}_{\color{red}\lrcorner}&{\color{blue}*}  &0& {\color{blue}\bf 1}\\
{\color{blue}*} &{\color{blue}*_1} &{\color{blue}*_1} &{\color{blue}\bf1} & 0&0&0&0&0\\
{\color{blue}*_3}&{\color{blue}*_1}&{\color{blue}*_1}& 0 & {\color{blue}*_2}_{\color{red}\lrcorner} & {\color{blue}\bf 1}&0&0&0\\
{\color{blue}*}&{\color{blue}*_1}&{\color{blue}*_1}& 0 &{\color{blue}*_2}&0&{\color{blue}\bf1}&0&0\\
{\color{blue}\bf1}&0&0_{\color{red}\lrcorner}&0&0_{\color{red}\lrcorner}&0&0&0& 0\\
0& {\color{blue}\bf 1}&0&0&0&0&0&0&0\\
0&0& {\color{blue}*}&{\color{blue}\bf1}&0&0&0&0&0\\
0&0&{\color{blue}\bf1}&0&0&0&0&0&0\
\end{array}\right].
\]
(The conditions imposed by intersecting with $\Omega_{[w]}$ say: the northwest $6\times 3$ submatrix has rank $\leq 2$; the northwest $6\times 5$ submatrix has rank $\leq 3$; the northwest $4\times 5$ submatrix has rank $\leq 2$; the northwest $2\times 6$ submatrix has rank $\leq 1$; and the northwest $1\times 7$ submatrix has rank $0$.  We will not need this in what follows, except to observe that these conditions are preserved by all operations.)

\begin{lemma}\label{l.typeAlabels}
There are exactly $|\mu|$ labelled entries.
\end{lemma}

\begin{proof}
From the definition of $\bk'$, there are $p_i-k'_i$ pivot $1$'s in the northwest $q_i\times p_i$ submatrix representing $X_v^\circ$.  So there are $q_i-p_i+k'_i = \mu_{k'_i}$ rows having no pivot.  And there are $k'_i$ columns having no pivot, of which $k'_{i-1}$ were labelled in previous steps.  So one sees $(k'_i-k'_{i-1})\mu_{k'_i}$ entries labelled `$i$'.  Summing over $i$ proves the lemma.
\end{proof}

Using the basis $(\e_1,0),\ldots,(\e_n,0),(0,\e_1),\ldots,(0,\e_n)$ for $V\oplus V$, matrix representatives for the embedded cell $\Sigma(X_v^\circ)$ have columns
\begin{align*}
 & (\pmb{c}_1,0),\ldots,(\pmb{c}_{p_1},0),(0,\e_{q_1+1}),\ldots,(0,\e_n), \\
 & (\pmb{c}_{p_1+1},0), \ldots,(\pmb{c}_{p_2},0),(0,\e_{q_2+1}),\ldots,(0,\e_{q_1}), \\
 & \cdots \\
 & (\pmb{c}_{p_s+1},0), \ldots,(\pmb{c}_{n},0),(0,\e_{1}),\ldots,(0,\e_{q_s}).
\end{align*}
These columns are separated into $s+1$ blocks, as indicated.  We will make the labels follow corresponding entries as they are embedded.

Continuing our example, the embedding in $Fl(\br;V\oplus V)$ is
\[\tiny
\Sigma(X_v^\circ) = \left[\begin{array}{cccccc|cc|cc|ccc|cc|ccc}
{\color{blue}*_3}&{\color{blue}*_1}&{\color{blue}*_1}&0&0&0&{\color{blue}*_4}&{\color{blue}*_2}&0&0&{\color{blue}*_4}&0&0&{\color{blue}*_5}&0&\color{blue}{\bf 1}&0&0\\
{\color{blue}*_3}&{\color{blue}*_1}&{\color{blue}*_1}&0&0&0&{\color{blue}*_4}&{\color{blue}*_2}&0&0&{\color{blue}*_4}&0&0& {\color{blue}*}&0&0&{\color{blue}\bf 1}&0\\
{\color{blue}*}&{\color{blue}*_1}&{\color{blue}*_1}&0&0&0&{\color{blue}\bf 1}&0&0&0&0&0&0&0&0&0&0&0\\
{\color{blue}*_3}&{\color{blue}*_1}&{\color{blue}*_1}&0&0&0&0&{\color{blue}*_2}&0&0& {\color{blue}\bf 1}&0&0&0&0&0&0&0\\
{\color{blue}*}&{\color{blue}*_1}&{\color{blue}*_1}&0&0&0&0&{\color{blue}*_2}&0&0&0&0&0&{\color{blue}\bf 1}&0&0&0&0\\
{\color{blue}\bf 1}&0&0&0&0&0&0&0&0&0&0&0&0&0&0&0&0&0\\
0& {\color{blue}\bf 1}&0&0&0&0&0&0&0&0&0&0&0&0&0&0&0&0\\
0&0& {\color{blue}*}&0&0&0&0&{\color{blue}\bf 1}&0&0&0&0&0&0&0&0&0&0\\
0&0&{\color{blue}\bf 1}&0&0&0&0&0&0&0&0&0&0&0&0&0&0&0\\\hline
0&0&0&0&0&0&0&0&0&0&0&0&0&0&0&0&0&\color{blue}{\bf 1}\\
0&0&0&0&0&0&0&0&0&0&0&0&0&0& {\color{blue}\bf 1}&0&0&0\\
0&0&0&0&0&0&0&0&0&0&0& {\color{blue}\bf 1}&0&0&0&0&0&0\\
0&0&0&0&0&0&0&0&0&0&0&0& {\color{blue}\bf 1}&0&0&0&0&0\\
0&0&0&0&0&0&0&0& {\color{blue}\bf 1}&0&0&0&0&0&0&0&0&0\\
0&0&0&0&0&0&0&0&0&{\color{blue}\bf 1}&0&0&0 &0&0&0&0&0\\
0&0&0& {\color{blue}\bf 1}&0&0&0&0&0&0&0&0&0&0&0&0&0&0\\
0&0&0&0& {\color{blue}\bf 1}&0&0&0&0&0&0&0&0&0&0&0&0&0\\
0&0&0&0&0& {\color{blue}\bf 1}&0&0&0&0&0&0&0&0&0&0&0&0
\end{array}
\right].
\]

Since we will impose conditions on intersections with $\Delta$, we change to a basis better adapted to the diagonal subspace.  With respect to the basis $(\e_1,0),\ldots,(\e_n,0),$
$(\e_1,\e_1),\ldots,$
$(\e_n,\e_n)$, the matrix representatives for $\Sigma(X_v^\circ)$ have columns
\begin{align*}
 & (\pmb{c}_1,0),\ldots,(\pmb{c}_{p_1},0),(-\e_{q_1+1},\e_{q_1+1}),\ldots,(-\e_n,\e_n), \\
 & (\pmb{c}_{p_1+1},0), \ldots,(\pmb{c}_{p_2},0),(-\e_{q_2+1},\e_{q_2+1}),\ldots,(-\e_{q_1},\e_{q_1}), \\
 & \cdots \\
 & (\pmb{c}_{p_s+1},0), \ldots,(\pmb{c}_{n},0),(-\e_{1},\e_{1}),\ldots,(-\e_{q_s},\e_{q_s}).
\end{align*}
In our running example, this is
\[\tiny
\left[\begin{array}{cccccc|cc|cc|ccc|cc|ccc}
{\color{blue}*_3}&{\color{blue}*_1}&{\color{blue}*_1}&0&0&0&{\color{blue}*_4}&{\color{blue}*_2}&0&0&{\color{blue}*_4}&0&0&{\color{blue}*_5}&0&\color{blue}{\bf 1}&0&\color{red}{\bf-1}\\
{\color{blue}*_3}&{\color{blue}*_1}&{\color{blue}*_1}&0&0&0&{\color{blue}*_4}&{\color{blue}*_2}&0&0&{\color{blue}*_4}&0&0& {\color{blue}*}&\color{red}{\bf-1}&0&{\color{blue}\bf 1}&0\\
{\color{blue}*}&{\color{blue}*_1}&{\color{blue}*_1}&0&0&0&{\color{blue}\bf 1}&0&0&0&0&\color{red}{\bf-1}&0&0&0&0&0&0\\
{\color{blue}*_3}&{\color{blue}*_1}&{\color{blue}*_1}&0&0&0&0&{\color{blue}*_2}&0&0& {\color{blue}\bf 1}&0&\color{red}{\bf-1}&0&0&0&0&0\\
{\color{blue}*}&{\color{blue}*_1}&{\color{blue}*_1}&0&0&0&0&{\color{blue}*_2}&\color{red}{\bf-1}&0&0&0&0&{\color{blue}\bf 1}&0&0&0&0\\
{\color{blue}\bf 1}&0&0&0&0&0&0&0&0&\color{red}{\bf-1}&0&0&0&0&0&0&0&0\\
0& {\color{blue}\bf 1}&0&\color{red}{\bf-1}&0&0&0&0&0&0&0&0&0&0&0&0&0&0\\
0&0& {\color{blue}*}&0&\color{red}{\bf-1}&0&0&{\color{blue}\bf 1}&0&0&0&0&0&0&0&0&0&0\\
0&0&{\color{blue}\bf 1}&0&0&\color{red}{\bf-1}&0&0&0&0&0&0&0&0&0&0&0&0\\\hline
0&0&0&0&0&0&0&0&0&0&0&0&0&0&0&0&0&\color{blue}{\bf 1}\\
0&0&0&0&0&0&0&0&0&0&0&0&0&0& {\color{blue}\bf 1}&0&0&0\\
0&0&0&0&0&0&0&0&0&0&0& {\color{blue}\bf 1}&0&0&0&0&0&0\\
0&0&0&0&0&0&0&0&0&0&0&0& {\color{blue}\bf 1}&0&0&0&0&0\\
0&0&0&0&0&0&0&0& {\color{blue}\bf 1}&0&0&0&0&0&0&0&0&0\\
0&0&0&0&0&0&0&0&0&{\color{blue}\bf 1}&0&0&0 &0&0&0&0&0\\
0&0&0& {\color{blue}\bf 1}&0&0&0&0&0&0&0&0&0&0&0&0&0&0\\
0&0&0&0& {\color{blue}\bf 1}&0&0&0&0&0&0&0&0&0&0&0&0&0\\
0&0&0&0&0& {\color{blue}\bf 1}&0&0&0&0&0&0&0&0&0&0&0&0
\end{array}
\right].
\]

Now we work with this matrix representation of $\Sigma(X_v^{\circ})$.  
Column operations within the $s+1$ blocks do not change the underlying partial flag in $Fl(\br;V\oplus V)$; nor do ``rightward'' column operations, which take a column from one block and add it to a column from a block to the right---these are precisely the operations coming from right-multiplication by  $\bar{P}_{\br}$.  (Sometimes we'll call these {\it admissible column operations}.)

Recall that $\Sigma(X_v^\circ) \cap \Omega_{w_\lambda^{-1}}$ is defined by imposing conditions on intersections with $\Delta$.  In our chosen basis,  these are equivalent to requiring that, for each $i$, the northwest $n\times r_i$ submatrix of $\Sigma(X_v^\circ)$ have rank at most $r_i-k_i$.  The subgroup $Q\subset GL_{2n}$ preserving $\Delta$ consists of block lower-triangular matrices, with two blocks of size $n$.  The action of $Q$ by left-multiplication (i.e., row operations) preserves $\Omega_{w_\lambda^{-1}}$.  
Since the conditions defining $\Omega_{w_\lambda^{-1}}$ concern only the first $n$ rows, from now on we focus on the top $n\times 2n$ submatrix.  We will use {\it admissible column operations} and {\it row operations from $Q$} to reduce $\Sigma(X_v^\circ)$ to echelon form.

Consider the $n\times 2n$ matrix representing $\Sigma(X_v^\circ)$, divided into $s+1$ blocks as before, so the $i$th block is on columns $r_{i-1}+1$ through $r_i$ (with the convention $r_0=0$).  We take this matrix to be generic, i.e., the $*$ entries are filled by independent variables.  In our example, it looks like this:
\begin{equation}\label{e.n2n}
\left[\begin{array}{cccccc|cc|cc|ccc|cc|ccc}
{\color{blue}*_3}&{\color{blue}*_1}&{\color{blue}*_1}&0&0&0&{\color{blue}*_4}&{\color{blue}*_2}&0&0&{\color{blue}*_4}&0&0&{\color{blue}*_5}&0&\color{blue}{\bf 1}&0&\color{red}{\bf -1}\\
{\color{blue}*_3}&{\color{blue}*_1}&{\color{blue}*_1}&0&0&0&{\color{blue}*_4}&{\color{blue}*_2}&0&0&{\color{blue}*_4}&0&0& {\color{blue}*}&{\color{red}\bf -1}&0&{\color{blue}\bf1}&0\\
{\color{blue}*}&{\color{blue}*_1}&{\color{blue}*_1}&0&0&0&{\color{blue}\bf1}&0&0&0&0&{\color{red}\bf -1}&0&0&0&0&0&0\\
{\color{blue}*_3}&{\color{blue}*_1}&{\color{blue}*_1}&0&0&0&0&{\color{blue}*_2}&0&0& {\color{blue}\bf 1}&0&{\color{red}\bf -1}&0&0&0&0&0\\
{\color{blue}*}&{\color{blue}*_1}&{\color{blue}*_1}&0&0&0&0& {\color{blue}*_2}&{\color{red}\bf -1}&0&0&0&0&{\color{blue}\bf1}&0&0&0&0\\
{\color{blue}\bf1}&0&0&0&0&0&0&0&0&{\color{red}\bf -1}&0&0&0&0&0&0&0&0\\
0& {\color{blue}\bf 1}&0&{\color{red}\bf -1}&0&0&0&0&0&0&0&0&0&0&0&0&0&0\\
0&0& {\color{blue}*}&0&{\color{red}\bf -1}&0&{\color{blue}\bf1}&0&0&0&0&0&0&0&0&0&0&0\\
0&0&{\color{blue}\bf1}&0&0&{\color{red}\bf -1}&0& 0&0&0&0&0&0&0&0&0&0&0\\
\end{array}\right]
\end{equation}

\begin{claim}\label{claim.rank.typeA}
 The submatrix formed by the first $i$ blocks---i.e., on the first $r_i$ columns---has rank at least $n-\mu_{k'_i}$. So block $i$ has minimal rank $ \mu_{k'_{i-1}}-\mu_{k'_i}$.  (And block $1$ has minimal rank $n-\mu_{k'_1}$.)
\end{claim}

\begin{proof}
Consider this $n \times r_i$ matrix.  By construction, each column has a ``pivot'' $1$ or $-1$ in it, so there are $r_i$ such entries.  Also by construction, there are exactly $k'_i$ rows which contain both a $1$ and a $-1$.  In such rows, mark the first (leftmost) of them as a pivot.  So there are $r_i - k'_i = p_i+n-q_i-k'_i = n-\mu_{k'_i}$ pivots.  Without changing the rank, we may rearrange and scale columns so that these pivots are all $1$'s and proceed NW to SE in the first $n-\mu_{k'_i}$ columns.  The square submatrix on these columns, and rows containing their pivots, is upper unitriangular.  The claim about the rank follows.
\end{proof}

Next we use row operations, and admissible column operations, on this $n\times 2n$ matrix to put it into echelon form, keeping track of where the free entries end up.  These operations preserve ranks of each block. So in its echelon form, this means the $i$th block has an identity matrix of size $\mu_{k'_{i-1}}-\mu_{k'_i}$ in its northwest corner, and its free entries must fit inside the complementary southeast corner, which has size $\mu_{k'_i} \cdot (k'_{i}-k'_{i-1})$.  We will see that, in fact, the entries labelled $i$ land in block $i$ of this echelon form.

That is, the reduced matrices all belong to $X_{w_\mu^{-1}}^\circ$, in its usual matrix form.  
In our example, $X_{w_\mu^{-1}}^\circ$ is this:
\begin{equation}\label{e.ech-form}
\left[\begin{array}{cccccc|cc|cc|ccc|cc|ccc}
{ 1}&0&0&0&0&0&0&0&0&0&0&{ 0}&0&0&0&0&0&0\\
0& { 1}&0&0&{ 0}&0&0&0&0&0&0&0&0&0&0&0&0&0\\
0&0& { 1}&0&0&{ 0}&0&0&0&0&0&0&0&0&0&0&0&0\\
0&0&0&{ 1}&0&0&0& { 0}&0&0&0&0&0&0&0&0&0&0\\
0&0&0&0&\color{red}{*_1}&\color{red}{*_1}& { 1}&0&0&{ 0}&0&0&0&0&0&0&0&0\\
0&0&0&0&\color{red}{*_1}&\color{red}{*_1}&0&\color{red}{*_2}&{ 1}&0&0&0&0&0&0&0&{ 0}&0\\
0&0&0&0&\color{red}{*_1}&\color{red}{*_1}&0&\color{red}{*_2}&0&\color{red}{*_3}& { 1}&0&{ 0}&0&0&0&0&0\\
0&0&0&0&\color{red}{*_1}&\color{red}{*_1}&0&\color{red}{*_2}&0&\color{red}{*_3}&0&\color{red}{*_4}&\color{red}{*_4}& { 1}&{ 0}&0&0&0\\
0&0&0&0&{\color{red}*_1}&{\color{red}*_1}&0&\color{red}{*_2}&0&{\color{red}*_3}&0&\color{red}{*_4}&\color{red}{*_4}&0&\color{red}{*_5}&{ 1}&0&{ 0}\\
\end{array}
\right].
\end{equation}

The key point now is that the labelled free entries of \eqref{e.n2n} are exactly those which survive in echelon form \eqref{e.ech-form}; the unlabelled ones are eliminated by upper-triangular row operations, accounting for the unipotent group $U$.

Before reduction, the free entries of \eqref{e.n2n} occur as one of the following possible types:
\begin{enumerate}[(a)]
   \item $\left[\begin{array}{ccc} {\color{blue}*} & \cdots & {\color{red}-1} \\ \vdots \\ {\color{blue}1}\end{array} \right]$ (within a block), or $\left[\begin{array}{cccccc} {\color{blue}*} & \cdots & \pm 1 & \cdots & \mp 1 \\ \vdots & &  & \vdots \\ {\color{blue}1} & \cdots & \cdots & {\color{red}-1} \end{array}\right] $; \medskip \\  or \medskip
        \item  $\left[\begin{array}{cccc} {\color{blue}*} & \cdots & \cdots & \pm 1 \\ \vdots & & \vdots \\ {\color{blue}1} & \cdots & {\color{red}-1} \end{array}\right] $ or $\left[\begin{array}{ccccc} & & {\color{blue}*} & \cdots & \pm 1 \\ & &  \vdots \\ {\color{red}-1} & \cdots & {\color{blue}1}\end{array} \right] $.
\end{enumerate}

Case (a) corresponds to unlabelled entries.  Here the free entry $*$ is eliminated by row operations, so such entries are absorbed into the unipotent subgroup $U$.  Carrying out these operations, we obtain an isomorphism
\[
  \Sigma(X_v^\circ) \isom U \times M,
\]
where $M$ is the locus of matrices in $\Sigma(X_v^\circ)$ having free entries only of type (b) (the labelled entries).

In case (b), the free entry is not eliminated by row operations, and survives in the reduced echelon form.  So there is an isomorphism $M \to M'$, given by left-multiplication by an element $g\in Q$, where $M'$ is a locus of matrices in echelon form, of the type identified above (as in \eqref{e.ech-form}).

So to complete the proof of Lemma~\ref{l.KLisom}, it remains to see that $M' = X_{w_\mu^{-1}}^\circ$. Since both are affine spaces, and $M'\subseteq X_{w_\mu^{-1}}^\circ$ by the above considerations, it suffices to compute $\dim M'$, that is, the number of inversions of type (b).

\begin{claim}\label{claimA2}
There are precisely $|\mu|$ free entries of type (b).
\end{claim}

Indeed, the entries of type (b) are exactly the labelled entries, so the claim follows from Lemma~\ref{l.typeAlabels}.

This completes the proof in type A.  Before turning to other types, we observe that the row operations used in the proof---left multiplication by elements of $Q$---actually come from the subgroup $GL(V) = GL(V\oplus 0) \subset Q \subset GL(V\oplus V)$.

\subsection{Type C}

The isomorphisms in other (classical) types are constructed exactly as in type A.  In fact, via standard embeddings of isotropic flag varieties in the usual (type A) flag varieties, the row operations described above for type A induce the required isomorphisms in other types.  We will spell this out in detail for type C; types D and B require only minor adjustments.

In outline, here is how we proceed.  The embedding $Fl^C(\bp;V) \hookrightarrow Fl(\tilde\bp;V)$ also determines an embedding of Schubert cells $X_v^\circ \hookrightarrow X_{\tilde{v}}^\circ$ such that $X_v^\circ = X_{\tilde{v}}^\circ \cap Fl^C(\bp;V)$ (usually not transversally), and similarly one has $\Omega_w \hookrightarrow \Omega_{\tilde{w}}$ with $\Omega_w = \Omega_{\tilde{w}} \cap Fl^C(\bp;V)$.

As noted in \S\ref{s.dirsum}, these embeddings are compatible with the direct sum map.  For a triple $\triple=(\bk,\bp,\bq)$ of type C, with corresponding strict partition $\lambda$, there is an extended triple $\tilde\triple=(\tilde\bk,\tilde\bp,\tilde\bq)$ of type A, with corresponding partition $\tilde\lambda$.  These have corresponding vexillary (signed) permutations $w$ and $\tilde{w}$.  For $v\geq w$, we have a type C weak triple $\triple' = (\bk',\bp,\bq)$ and an extension to a type A weak triple $\tilde\triple' = (\tilde\bk',\tilde\bp,\tilde\bq)$.  We will consider matrix representatives for $\Sigma(X_v^\circ) \subset \Sigma(X_{\tilde{v}}^\circ)$ inside $Fl^C(\br;V\oplus V) \subset Fl(\tilde{\br};V\oplus V)$.  The isomorphism of the type A Schubert cells $\Sigma(X_{\tilde{v}}^\circ) \isom \tilde{U} \times X_{w_{\tilde{\mu}}^{-1}}$, described above, induces the required isomorphism $\Sigma(X_{v}^\circ) \isom U \times X_{w_{\mu}^{-1}}$.  As before, this comes from row operations coming from left multiplication by a copy of $GL(V) \subset GL(V\oplus V)$.  Using an appropriate choice of basis---adapted to the diagonal subspace $\Delta\subset V\oplus V$, as before---this copy of $GL(V)$ lies inside the subgroup preserving the bilinear form $\langle\!\langle \;, \; \rangle\!\rangle$ on $V\oplus V$.

\medskip
Now let us spell out the details.  Given a type C triple $\triple=(\bk,\bp,\bq)$, the corresponding strict partition has $\lambda_{k_i} = p_i+q_i-1$, with the other parts filled in minimally so that $\lambda_1>\cdots>\lambda_s$.  Similarly, given $\triple'=(\bk',\bp,\bq)$, we have $\mu_{k'_i} = p_i+q_i-1$.  We set $r_i = p_i+q_i-1$.  The target of the direct sum map is $Fl^C(\br;V\oplus V) = \bar{G}/\bar{P}_{\br}$, where $\bar{G} = Sp_{4n}$ and $P_{\br}$ is the block-upper-triangular matrix preserving a standard isotropic flag.

As in type A, we start by choosing an appropriate basis.  Write $\e_{\bar{n}},\ldots,\e_{\bar1},\e_1,\ldots,\e_n$ for the standard basis of $V$, as in our conventions for type C, so $F_{q_i}$ is the span of $\e_n,\ldots,\e_{q_i}$.  The standard $2n\times 2n$ matrix representatives for $X_v^\circ$ are given by the $B^+$-orbit of the permutation matrix for the signed permutation $v$ (extended to a permutation in $S_{2n}$).  Such representatives have $1$'s in positions $(v(j),j)$ (for $j\in \{\bar{n},\ldots,\bar{1},1,\ldots,n\}$, and in positions $(i,j)$ such that $v(j)>i$ and $v^{-1}(i)>j$, the entries are either free (written $*$) or constrained by the isotropic condition (written $\bullet$).  Elsewhere there are $0$'s.  See, e.g., \cite[\S6]{FP} for this way of parametrizing Schubert cells (but note the conventions there are different from ours).

These matrix representatives naturally embed $X_v^\circ \subset Fl^C(\bp;V) \subset Fl(\tilde\bp;V)$, where
\[
 \tilde{p}_i = \begin{cases} n+1-p_i & \text{for }1\leq i\leq s; \\
                             n-1+p_{2s+1-i} & \text{for }s+1\leq i\leq 2s. \end{cases}
\]
These numbers are the dimensions of $E_p$ and $E_p^\perp$.  The labelling of $*$ and $\bullet$ entries is done just as in type A.

For our running example, take $n=5$, and let $\triple = (\bk,\bp,\bq) = (1\,2,\; 3\,1,\; 3\,2)$, so $w = \bar{2}\,1\,\bar{3}\,4\,5$.  With $v= 1\,3\,\bar{5}\,\bar{4}\,\bar{2}$, we find $\bk' = (2\,3)$, so
\[
  \triple' = (\bk',\bp,\bq) = (2\,3,\; 3\,1,\; 3\,2)
\]
and $\mu = (6,5,2)$.  The isotropic flag $E_3 \subset E_1 \subset V$ extends to $E_3 \subset E_1 = E_1^\perp \subset E_3^\perp \subset V$, and $\tilde\bp = (3\,5\,5\,7)$.  Matrix representatives are 
\[ 
X_v^\circ= \left[
\begin{array}{ccc|cc|cc|ccc}
{\color{blue}*_2}&{\color{blue}*_1}&{\color{blue}*_1}&{\color{blue}\bullet}&{\color{blue}\bullet_3}&{\color{blue}\bullet_4}&{\color{blue}\bullet_4}&{\color{blue}\bf1}&0& 0 \\
{\color{blue}*_2}&{\color{blue}*_1}&{\color{blue}\bullet_1}&{\color{blue}\bullet}&{\color{blue}\bullet_3}&{\color{blue}\bullet_4}&{\color{blue}\bullet_4}  & 0 & {\color{blue}\bf 1} & 0 \\
{\color{blue}*} &{\color{blue}*_1} &{\color{blue}*_1} &{\color{blue}\bf1} & 0&0&0_{\color{red}\lrcorner}&0&0 &0 \\
{\color{blue}*_2}&{\color{blue}\bullet_1}&{\color{blue}\bullet_1}& 0 & {\color{blue}\bullet_3}_{\color{red}\lrcorner} &  {\color{blue}\bullet} &  {\color{blue}\bullet} & 0 & 0 & {\color{blue}\bf 1}\\ 
{\color{blue}*}&{\color{blue}*_1}&{\color{blue}*_1}& 0&{\color{blue}\bf1}&0&0 &0 &0 &0\\ \hline 
{\color{blue}*_2}&{\color{blue}*_1}&{\color{blue}*_1}&0&0_{\color{red}\lrcorner}&{\color{blue}\bf1}&0&0&0 &0\\
{\color{blue}\bf1}&0&0_{\color{red}\lrcorner}&0&0&0&0&0& 0 & 0\\
0& {\color{blue}*}&{\color{blue}*}&0&0&0&{\color{blue}\bf1}&0&0&0\\
0&{\color{blue}\bf1}&0&0&0&0&0&0&0&0\\
0&0&{\color{blue}\bf1}&0&0&0&0&0&0&0
\end{array} \right].
\]
(The conditions imposed by intersecting with $\Omega_{[w]}$, namely (1) $\dim(E_3 \cap F_3)\geq 1$ and (2) $\dim(E_1\cap F_2)\geq 2$, say that (1) the northwest $7\times 3$ submatrix has rank at most $3-1=2$, and (2) the northwest $6\times 5$ submatrix has rank at most $5-2=3$.  Again, we will not need this in the proof, except to note these conditions are preserved.)

Given $\triple' = (\bk',\bp,\bq)$, the extension  $\tilde\triple' = (\tilde\bk',\tilde\bp,\tilde\bq)$ is defined by
\[
 \tilde{q}_i = \begin{cases} n-1+q_i & \text{for }1\leq i\leq s; \\
                             n+1-q_{2s+1-i} & \text{for }s+1\leq i\leq 2s; \end{cases}
\]
and
\[
 \tilde{k}'_i =  \begin{cases} k'_i & \text{for }1\leq i\leq s; \\
                             p_{2s+1-i}+q_{2s+1-i} + k'_{2s+1-i}-2 & \text{for }s+1\leq i\leq 2s. \end{cases}
\]
The numbers $\tilde{k}'_i$ are computed using $(E_p\cap F_q)^\perp = E_p^\perp+F_q^\perp$ for isotropic spaces $E_p$ and $F_q$.  So if $\dim(E_p \cap F_q) = k$, then $\dim( E_p^\perp + F_q^\perp ) = 2n-k$; together with the formulas for $\dim(E_p^\perp)$ and $\dim(F_q^\perp)$, this implies $\dim( E_p^\perp \cap F_q^\perp ) = p+q+k-2$.

When $p_s=q_s=1$, the formulas produce $(\tilde{k}'_{s+1}, \tilde{p}_{s+1}, \tilde{q}_{s+1}) = (\tilde{k}'_s,\tilde{p}_s,\tilde{q}_s)$ and we usually omit this repetition.  The partition $\tilde\mu$ is the one associated to the (type A) triple $\tilde\triple'$.

In our running example, $\triple' = (2\,3,\;3\,1,\;3\,2)$, and $\tilde\triple' = (2\,3\,4\,6,\; 3\,5\,5\,7,\;7\,6\,4\,3)$, so $\tilde\mu = (6,6,4,3,2,2)$.

There is some flexibility in the placement of $\bullet$ entries---that is, in deciding which entries are dependent on the others.  We will exchange $\bullet_i$ in position $(a,b)$ with $*_j$ in position $(a',b')$ whenever $i<j$, $v(b)=a'$, and $v(b')=a$.  

\[ 
X_v^\circ= \left[
\begin{array}{ccc|cc|cc|ccc}
{\color{blue}\bullet_2}&{\color{blue}*_1}&{\color{blue}*_1}&{\color{blue}\bullet}&{\color{blue}\bullet_3}&{\color{blue}\bullet_4}&{\color{blue}\bullet_4}&{\color{blue}\bf1}&0& 0 \\
{\color{blue}\bullet_2}&{\color{blue}*_1}&{\color{blue}\bullet_1}&{\color{blue}\bullet}&{\color{blue}\bullet_3}&{\color{blue}\bullet_4}&{\color{blue}\bullet_4}  & 0 & {\color{blue}\bf 1} & 0 \\
{\color{blue}*} &{\color{blue}*_1} &{\color{blue}*_1} &{\color{blue}\bf1} & 0&0&0_{\color{red}\lrcorner}&0&0 &0 \\
{\color{blue}*_2}&{\color{blue}*_1}&{\color{blue}*_1}& 0 & {\color{blue}\bullet_3}_{\color{red}\lrcorner} &  {\color{blue}\bullet} &  {\color{blue}\bullet} & 0 & 0 & {\color{blue}\bf 1}\\ 
{\color{blue}*}&{\color{blue}*_1}&{\color{blue}*_1}& 0&{\color{blue}\bf1}&0&0 &0 &0 &0\\ \hline 
{\color{blue}*_2}&{\color{blue}*_1}&{\color{blue}*_1}&0&0_{\color{red}\lrcorner}&{\color{blue}\bf1}&0&0&0 &0\\
{\color{blue}\bf1}&0&0_{\color{red}\lrcorner}&0&0&0&0&0& 0 & 0\\
0& {\color{blue}*}&{\color{blue}*}&0&0&0&{\color{blue}\bf1}&0&0&0\\
0&{\color{blue}\bf1}&0&0&0&0&0&0&0&0\\
0&0&{\color{blue}\bf1}&0&0&0&0&0&0&0
\end{array} \right].
\]

\begin{lemma}\label{l.typeClabels}
There are $|\mu|$ labelled free entires.
\end{lemma}

\begin{proof}
The proof is similar to type A.  In fact, Lemma~\ref{l.typeAlabels} shows there are $|\tilde\mu|$ labelled entries $*$ and $\bullet$.  The isotropic conditions account for the $\bullet$'s: listing $k_s$ columns in any order compatible with the labels (so columns with $*_1$ come first, then $*_2$, etc.), there are $k-1$ entries $\bullet$ in the $k$th column.  For $1\leq i\leq s$, then, there are
\[
  \sum_{k=k_{i-1}+1}^{k_i} \left(\tilde\mu_k - (k-1)\right) = \sum_{k=k_{i-1}+1}^{k_i} \mu_k
\]
free entries $*_i$.  Summing over $i$ proves the lemma.
\end{proof}

Using the basis $(\e_{\bar{n}},0),\ldots,(\e_n,0),(0,\e_{\bar{n}}),\ldots,(0,\e_n)$ for $V\oplus V$, matrix representatives for $\Sigma(X_v^\circ) \subset Fl^C(\br;V\oplus V) \subset Fl(\tilde\br;V\oplus V)$ are given by a $4n\times 4n$ matrix, similar to type A.  In our example, this is
\[\tiny
\Sigma(X_v^\circ) = \left[\begin{array}{cccccc|ccc|c||c|ccc|cccccc}
{\color{blue}\bullet}&{\color{blue}*}&{\color{blue}*}&0&0&0&{\color{blue}\bullet}&{\color{blue}\bullet}&0&0&0&0&{\color{blue}\bullet}&{\color{blue}\bullet}&0&0&0&\color{blue}{\bf 1}&0&0\\
{\color{blue}\bullet}&{\color{blue}*}&{\color{blue}\bullet}&0&0&0&{\color{blue}\bullet}&{\color{blue}\bullet}&0&0&0&0&{\color{blue}\bullet}&{\color{blue}\bullet}&0&0&0&0&{\color{blue}\bf 1}&0\\
{\color{blue}*}&{\color{blue}*}&{\color{blue}*}&0&0&0&{\color{blue}\bf 1}&0&0&0&0&0&0&0&0&0&0&0&0&0\\
{\color{blue}*}&{\color{blue}*}&{\color{blue}*}&0&0&0&0&{\color{blue}\bullet}&0&0&0&0&{\color{blue}\bullet}&{\color{blue}\bullet}&0&0&0&0&0& {\color{blue}\bf 1}\\
{\color{blue}*}&{\color{blue}*}&{\color{blue}*}&0&0&0&0&{\color{blue}\bf 1}&0&0&0&0&0&0&0&0&0&0&0&0\\ \hline
{\color{blue}*}&{\color{blue}*}&{\color{blue}*}&0&0&0&0&0&0&0&0&0&{\color{blue}\bf1}&0&0&0&0&0&0&0\\
{\color{blue}\bf 1}&0&0&0&0&0&0&0&0&0&0&0&0&0&0&0&0&0&0&0\\
0&{\color{blue}*}&{\color{blue}*}&0&0&0&0&0&0&0&0&0&0&{\color{blue}\bf1}&0&0&0&0&0&0\\
0&{\color{blue}\bf 1}&0&0&0&0&0&0&0&0&0&0&0&0&0&0&0&0&0&0\\
0&0&{\color{blue}\bf 1}&0&0&0&0&0&0&0&0&0&0&0&0&0&0&0&0&0\\ \hline\hline
0&0&0&0&0&0&0&0&0&0&0&0&0&0&\color{blue}{\bf 1}&0&0&0&0&0\\
0&0&0&0&0&0&0&0&0&0&0&0&0&0&0&\color{blue}{\bf 1}&0&0&0&0\\
0&0&0&0&0&0&0&0&0&0&0&0&0&0&0&0& {\color{blue}\bf 1}&0&0&0\\
0&0&0&0&0&0&0&0&0&0&0& {\color{blue}\bf 1}&0&0&0&0&0&0&0&0\\
0&0&0&0&0&0&0&0&0&{\color{blue}\bf 1}&0&0&0&0&0&0&0&0&0&0\\ \hline
0&0&0&0&0&0&0&0&0&0& {\color{blue}\bf 1}&0&0&0&0&0&0&0&0&0\\
0&0&0&0&0&0&0&0&{\color{blue}\bf 1}&0&0&0 &0&0&0&0&0&0&0&0\\
0&0&0& {\color{blue}\bf 1}&0&0&0&0&0&0&0&0&0&0&0&0&0&0&0&0\\
0&0&0&0& {\color{blue}\bf 1}&0&0&0&0&0&0&0&0&0&0&0&0&0&0&0\\
0&0&0&0&0& {\color{blue}\bf 1}&0&0&0&0&0&0&0&0&0&0&0&0&0&0
\end{array}
\right].
\]
(Lines are added as visual aids, to distinguish both the blocks corresponding to the partial flag variety $Fl(\tilde\br;V\oplus V)$ as before, as well as the axes of symmetry arising from the isotropic conditions.)

Next, we change basis to work with one which is both adapted to $\Delta\subset V\oplus V$, and with respect to which the bilinear form $\langle\!\langle\;,\;\rangle\!\rangle$ has antidiagonal Gram matrix.  A convenient choice is
\begin{equation}\label{e.adapted-basis-C}
 \frac{1}{2}(\e_{\bar{n}},-\e_{\bar{n}}), \ldots,\frac{1}{2}(\e_n,-\e_n), (\e_{\bar{n}},\e_{\bar{n}}), \ldots, (\e_n,\e_n).
\end{equation}
(Here we require $\mathrm{char}(\KK)\neq 2$.)  In this basis, the top half of the matrix for $\Sigma(X_v^\circ)$ has a similar form to the one described above in type A.  In our example, it is
\[\small
 \left[\begin{array}{cccccc|ccc|c||c|ccc|cccccc}
{\color{blue}\bullet}&{\color{blue}*}&{\color{blue}*}&0&0&0&{\color{blue}\bullet}&{\color{blue}\bullet}&0&0&0&0&{\color{blue}\bullet}&{\color{blue}\bullet}&\color{red}{\bf -1}&0&0&\color{blue}{\bf 1}&0&0\\
{\color{blue}\bullet}&{\color{blue}*}&{\color{blue}\bullet}&0&0&0&{\color{blue}\bullet}&{\color{blue}\bullet}&0&0&0&0&{\color{blue}\bullet}&{\color{blue}\bullet}&0&\color{red}{\bf -1}&0&0&{\color{blue}\bf 1}&0\\
{\color{blue}*}&{\color{blue}*}&{\color{blue}*}&0&0&0&{\color{blue}\bf 1}&0&0&0&0&0&0&0&0&0&\color{red}{\bf -1}&0&0&0\\
{\color{blue}*}&{\color{blue}*}&{\color{blue}*}&0&0&0&0&{\color{blue}\bullet}&0&0&0&\color{red}{\bf -1}&{\color{blue}\bullet}&{\color{blue}\bullet}&0&0&0&0&0& {\color{blue}\bf 1}\\
{\color{blue}*}&{\color{blue}*}&{\color{blue}*}&0&0&0&0&{\color{blue}\bf 1}&0&\color{red}{\bf -1}&0&0&0&0&0&0&0&0&0&0\\ \hline
{\color{blue}*}&{\color{blue}*}&{\color{blue}*}&0&0&0&0&0&0&0&\color{red}{\bf -1}&0&{\color{blue}\bf1}&0&0&0&0&0&0&0\\
{\color{blue}\bf 1}&0&0&0&0&0&0&0&\color{red}{\bf -1}&0&0&0&0&0&0&0&0&0&0&0\\
0&{\color{blue}*}&{\color{blue}*}&\color{red}{\bf -1}&0&0&0&0&0&0&0&0&0&{\color{blue}\bf1}&0&0&0&0&0&0\\
0&{\color{blue}\bf 1}&0&0&\color{red}{\bf -1}&0&0&0&0&0&0&0&0&0&0&0&0&0&0&0\\
0&0&{\color{blue}\bf 1}&0&0&\color{red}{\bf -1}&0&0&0&0&0&0&0&0&0&0&0&0&0&0\\ \hline\hline
\end{array}
\right].
\]
The bottom half of such a matrix has many more nonzero entries than the one we used in type A, but only the top half will be relevant.

The key features of this basis are:
\begin{itemize}
    \item The diagonal subspace $\Delta$ is spanned by the last $2n$ vectors of the basis, so conditions on $\dim( G_{r_i} \cap \Delta )$ are equivalent to rank conditions on the first $2n$ rows of the matrix representatives.
    
    \item The first $2n$ vectors span a subspace $\Delta^* \subset V\oplus V$ which is also isotropic with respect to $\langle\!\langle\;,\;\rangle\!\rangle$, so a copy of $GL_{2n} \subset Sp_{4n}$ acts by (arbitrary) row operations on the first $2n$ rows (combined with compensating simultaneous row operations on the last $2n$ rows).
\end{itemize}
In our example, the conditions imposed by intersecting with $\Omega_{[w_{\lambda}^{-1}]}$ say that the northwest $10\times 6$ submatrix has rank at most $6-1=5$, and the northwest $10\times 9$ submatrix has rank at most $9-2=7$.

With this in mind, together with the observation made after the conclusion of the type A argument, we may preform the same row operations as in type A, and end with matrix representatives lying in $\Sigma(X_{w_{\tilde\mu}^{-1}}^\circ)$.  But since each row operation preserves $\langle\!\langle\;,\;\rangle\!\rangle$, the result in fact lies in the subset $\Sigma(X_{w_\mu^{-1}}^\circ) \subseteq \Sigma(X_{w_{\tilde\mu}^{-1}}^\circ)$.  In our running example, these are matrices of the following form:
\[\small
 \left[\begin{array}{cccccc|ccc|c||c|ccc|cccccc}
1&0&0&0&0&0&0&0&0&0&0&0&0&0&0&0&0&0&0&0\\
0&1&0&0&0&0&0&0&0&0&0&0&0&0&0&0&0&0&0&0\\
0&0&1&0&0&0&0&0&0&0&0&0&0&0&0&0&0&0&0&0\\
0&0&0&1&0&0&0&0&0&0&0&0&0&0&0&0&0&0&0&0\\
0&0&0&0&{\color{red}*}&{\color{red}*}&1&0&0&0&0&0&0&0&0&0&0&0&0&0\\ \hline
0&0&0&0&{\color{red}*}&{\color{red}*}&0&1&0&0&0&0&0&0&0&0&0&0&0&0\\
0&0&0&0&{\color{red}*}&{\color{red}*}&0&0&{\color{red}*}&1&0&0&0&0&0&0&0&0&0&0\\
0&0&0&0&{\color{red}*}&{\color{red}*}&0&0&{\color{red}*}&0&{\color{red}\bullet}&1&0&0&0&0&0&0&0&0\\
0&0&0&0&{\color{red}*}&{\color{red}*}&0&0&{\color{red}\bullet}&0&{\color{red}\bullet}&0&{\color{red}\bullet}&{\color{red}\bullet}&1&0&0&0&0&0\\
0&0&0&0&{\color{red}*}&{\color{red}\bullet}&0&0&{\color{red}\bullet}&0&{\color{red}\bullet}&0&{\color{red}\bullet}&{\color{red}\bullet}&0&1&0&0&0&0\\ \hline\hline
\end{array}
\right].
\]

To complete the proof of Lemma~\ref{l.KLisom} in type C, we observe that a counting argument analogous to the one used in type A shows that all such matrices appear.  In addition to the free entries $*$ of type (a) and (b), which appear in configurations exactly as in type A, the $\bullet$ entries appear in two types:
\begin{enumerate}[(a)]
   \item[(a')] $\left[\begin{array}{ccc} {\color{blue}\bullet} & \cdots & {\color{red}-1} \\ \vdots \\ {\color{blue}1}\end{array} \right]$ (within a block), or $\left[\begin{array}{cccccc} {\color{blue}\bullet} & \cdots & \pm 1 & \cdots & \mp 1 \\ \vdots & &  & \vdots \\ {\color{blue}1} & \cdots & \cdots & {\color{red}-1} \end{array}\right] $; \medskip \\  or \medskip
   \item[(b')]  $\left[\begin{array}{cccc} {\color{blue}\bullet} & \cdots & \cdots & \pm 1 \\ \vdots & & \vdots \\ {\color{blue}1} & \cdots & {\color{red}-1} \end{array}\right] $ or $\left[\begin{array}{ccccc} & & {\color{blue}\bullet} & \cdots & \pm 1 \\ & &  \vdots \\ {\color{red}-1} & \cdots & {\color{blue}1}\end{array} \right] $.
\end{enumerate}
Just as in type A, row operations eliminating type (a) (and (a')) entries are absorbed into the unipotent group $U$.  Claim~\ref{claimA2} shows that there are $|\tilde\mu|$ entries of types (b) and (b'), combined, and that these are precisely the labelled entries.  Then Lemma~\ref{l.typeClabels} shows that there are $|\mu|$ (free) entries of type (b), as required.

\subsection{Type D}

Only a few changes are required to modify the type C argument into one which works for type D.  Given a type D triple $\triple = (\bk,\bp,\bq)$, the corresponding partition has $\lambda_{k_i} = p_i+q_i$, and similarly, $\triple'=(\bk',\bp,\bq)$ has partition with $\mu_{k'_i} = p_i+q_i$.  The extension of $\triple'$ is $\tilde\triple' = (\tilde\bk',\tilde\bp,\tilde\bq)$, with
\[
  \tilde{p}_i = \begin{cases} n-p_i &\text{for }1\leq i\leq s; \\ n+p_{2s+1-i} &\text{for }s+1\leq i\leq 2s; \end{cases}
\]
\[
  \tilde{q}_i = \begin{cases} n+q_i &\text{for }1\leq i\leq s; \\ n-q_{2s+1-i} &\text{for }s+1\leq i\leq 2s; \end{cases}
\]and
\[
  \tilde{k}'_i = \begin{cases} k'_i &\text{for }1\leq i\leq s; \\ p_{2s+1-i}+q_{2s+1-i}+k'_{2s+1-i} & \text{for } s+1\leq i\leq 2s. \end{cases}
\]
The reasons for these numbers are the same as in type C.  In the case $p_s=q_s=0$, the formulas produce $(\tilde{k}'_{s+1}, \tilde{p}_{s+1}, \tilde{q}_{s+1}) = (\tilde{k}'_s,\tilde{p}_s,\tilde{q}_s)$ and as in type C, we usually omit this repetition.

The matrix manipulations are essentially the same as in type C.  We briefly illustrate with an example, for $n=5$.  Take $\triple = (1\,2,\; 2\,0,\; 2\,1)$, so $w=\bar{2}\,1\,\bar{3}\,4\,5$ and $\lambda = (4,1)$.  With $v=\bar{1}\,3\,\bar{5}\,\bar{4}\,\bar{2}$, we find $k_v(q_1,p_1)=2$ and $k_v(q_2,p_2)=3$.  So we get a weak triple $\triple' = (\bk',\bp,\bq) = (2\,3,\; 2\,0,\; 2\,1)$, with $\mu = (5,4,1,0)$.  Then $\tilde\triple' = (2\,3\,4\,6,\; 3\,5\,5\,7,\; 7\,6\,4\,3)$, with $\tilde\mu = (6,6,4,3,2,2))$.

In matrices,
\[ 
X_v^\circ= \left[
\begin{array}{ccc|cc|cc|ccc}
{\color{blue}*}&{\color{blue}*}&{\color{blue}\bullet}&{\color{blue}\bullet}&{\color{blue}\bullet}&{\color{blue}\bullet}&{\color{blue}\bullet}&{\color{blue}\bf1}&0& 0 \\
{\color{blue}*}&{\color{blue}\bullet}&{\color{blue}\bullet}&{\color{blue}\bullet}&{\color{blue}\bullet}&{\color{blue}\bullet}&{\color{blue}\bullet}  & 0 & {\color{blue}\bf 1} & 0 \\
{\color{blue}*} &{\color{blue}*} &{\color{blue}*} &{\color{blue}\bf1} & 0&0&0&0&0 &0 \\
{\color{blue}\bullet}&{\color{blue}\bullet}&{\color{blue}\bullet}& 0 & {\color{blue}\bullet} &  {\color{blue}\bullet} &  {\color{blue}\bullet} & 0 & 0 & {\color{blue}\bf 1}\\ 
{\color{blue}*}&{\color{blue}*}&{\color{blue}*}& 0&0&{\color{blue}\bf1}&0 &0 &0 &0\\ \hline 
{\color{blue}*}&{\color{blue}*}&{\color{blue}*}&0&{\color{blue}\bf1}&0&0&0&0 &0\\
{\color{blue}\bf1}&0&0&0&0&0&0&0& 0 & 0\\
0& {\color{blue}*}&{\color{blue}*}&0&0&0&{\color{blue}\bf1}&0&0&0\\
0&{\color{blue}\bf1}&0&0&0&0&0&0&0&0\\
0&0&{\color{blue}\bf1}&0&0&0&0&0&0&0
\end{array} \right].
\]
Under the direct sum embedding, after re-arranging the $\bullet$ entries and changing to the basis \eqref{e.adapted-basis-C}, the top half of the matrix representing $\Sigma(X_v^\circ)$ is
\[\small
\left[\begin{array}{cccccc|ccc|c||c|ccc|cccccc}
{\color{blue}\bullet}&{\color{blue}*}&{\color{blue}\bullet}&0&0&0&{\color{blue}\bullet}&{\color{blue}\bullet}&0&0&0&0&{\color{blue}\bullet}&{\color{blue}\bullet}&\color{red}{\bf -1}&0&0&\color{blue}{\bf 1}&0&0\\
{\color{blue}\bullet}&{\color{blue}\bullet}&{\color{blue}\bullet}&0&0&0&{\color{blue}\bullet}&{\color{blue}\bullet}&0&0&0&0&{\color{blue}\bullet}&{\color{blue}\bullet}&0&\color{red}{\bf -1}&0&0&{\color{blue}\bf 1}&0\\
{\color{blue}*}&{\color{blue}*}&{\color{blue}*}&0&0&0&{\color{blue}\bf 1}&0&0&0&0&0&0&0&0&0&\color{red}{\bf -1}&0&0&0\\
{\color{blue}\bullet}&{\color{blue}*}&{\color{blue}*}&0&0&0&0&{\color{blue}\bullet}&0&0&0&\color{red}{\bf -1}&{\color{blue}\bullet}&{\color{blue}\bullet}&0&0&0&0&0& {\color{blue}\bf 1}\\
{\color{blue}*}&{\color{blue}*}&{\color{blue}*}&0&0&0&0&0&0&0&\color{red}{\bf -1}&0&{\color{blue}\bf 1}&0&0&0&0&0&0&0\\ \hline
{\color{blue}*}&{\color{blue}*}&{\color{blue}*}&0&0&0&0&{\color{blue}\bf1}&0&\color{red}{\bf -1}&0&0&0&0&0&0&0&0&0&0\\
{\color{blue}\bf 1}&0&0&0&0&0&0&0&\color{red}{\bf -1}&0&0&0&0&0&0&0&0&0&0&0\\
0&{\color{blue}*}&{\color{blue}*}&\color{red}{\bf -1}&0&0&0&0&0&0&0&0&0&{\color{blue}\bf1}&0&0&0&0&0&0\\
0&{\color{blue}\bf 1}&0&0&\color{red}{\bf -1}&0&0&0&0&0&0&0&0&0&0&0&0&0&0&0\\
0&0&{\color{blue}\bf 1}&0&0&\color{red}{\bf -1}&0&0&0&0&0&0&0&0&0&0&0&0&0&0\\ \hline\hline
\end{array}
\right].
\]
After performing row operations, the reduced form is
\[\small
 \left[\begin{array}{cccccc|ccc|c||c|ccc|cccccc}
1&0&0&0&0&0&0&0&0&0&0&0&0&0&0&0&0&0&0&0\\
0&1&0&0&0&0&0&0&0&0&0&0&0&0&0&0&0&0&0&0\\
0&0&1&0&0&0&0&0&0&0&0&0&0&0&0&0&0&0&0&0\\
0&0&0&1&0&0&0&0&0&0&0&0&0&0&0&0&0&0&0&0\\
0&0&0&0&{\color{red}*}&{\color{red}*}&1&0&0&0&0&0&0&0&0&0&0&0&0&0\\ \hline
0&0&0&0&{\color{red}*}&{\color{red}*}&0&1&0&0&0&0&0&0&0&0&0&0&0&0\\
0&0&0&0&{\color{red}*}&{\color{red}*}&0&0&{\color{red}*}&1&0&0&0&0&0&0&0&0&0&0\\
0&0&0&0&{\color{red}*}&{\color{red}*}&0&0&{\color{red}\bullet}&0&{\color{red}\bullet}&1&0&0&0&0&0&0&0&0\\
0&0&0&0&{\color{red}*}&{\color{red}\bullet}&0&0&{\color{red}\bullet}&0&{\color{red}\bullet}&0&{\color{red}\bullet}&{\color{red}\bullet}&1&0&0&0&0&0\\
0&0&0&0&{\color{red}\bullet}&{\color{red}\bullet}&0&0&{\color{red}\bullet}&0&{\color{red}\bullet}&0&{\color{red}\bullet}&{\color{red}\bullet}&0&1&0&0&0&0\\ \hline\hline
\end{array}
\right].
\]

The proof of the type D version of Lemma~\ref{l.KLisom} again follows from the analogous counting argument.

\subsection{Type B}

Here there is nothing new, apart from the formulas for the extended triple.  Given triples $\triple=(\bk,\bp,\bq)$ and $\triple'=(\bk',\bp,\bq)$ (of type B, which is the same as C), we extend $\triple'$ by
\[
  \tilde{p}_i = \begin{cases} n+1-p_i & \text{for } 1\leq i\leq s; \\ n+p_{2s+1-i} & \text{for } s+1\leq i\leq 2s; \end{cases}
\]
\[
  \tilde{q}_i = \begin{cases} n+q_i & \text{for } 1\leq i\leq s; \\ n+1-q_{2s+1-i} & \text{for } s+1\leq i\leq 2s; \end{cases}
\]
and
\[
  \tilde{k}'_i = \begin{cases} k'_i & \text{for } 1\leq i\leq s; \\ p_{2s+1-i}+q_{2s+1-i}+k'_{2s+1-i}-1 & \text{for } s+1\leq i\leq 2s. \end{cases}
\]

\begin{remark}
In \cite{WY}, Woo and Yong introduce the notion of ``pattern interval embedding'' as a tool for comparing singularities of Schubert varieties lying in different flag varieties.  The underlying geometry of their method uses isomorphisms between Richardson varieties.  By contrast, even in the case where $\ell(v)=|\mu|$, the isomorphism of our Theorem~\ref{t.KLisom} does not extend to one between Richardson varieties $\Omega_w \cap X_v$ and $\Omega_{w_\lambda^{-1}} \cap X_{w_\mu^{-1}}$.  In fact, there are many examples where the Bruhat intervals $[w,v]$ and $[w_\lambda^{-1},w_\mu^{-1}]$ are not equinumerous.
\end{remark}


\end{document}